\documentclass[11pt,a4paper]{amsart}
\usepackage{marginnote}
\usepackage[pdftex]{graphicx}
\usepackage[pdftex]{xcolor}
\usepackage{ascmac}
\usepackage{enumitem}  
\usepackage{graphicx}
\usepackage{hyperref}
\usepackage{mathtools,  stmaryrd}
\usepackage{accents}
\usepackage{xparse} \DeclarePairedDelimiterX{\Iintv}[1]{\llbracket}{\rrbracket}{\iintvargs{#1}}
\NewDocumentCommand{\iintvargs}{>{\SplitArgument{1}{,}}m}
{\iintvargsaux#1} %
\NewDocumentCommand{\iintvargsaux}{mm} {#1\mkern1.5mu,\mkern1.5mu#2}

\usepackage[normalem]{ulem}

\makeatletter
\newtheorem*{rep@theorem}{\rep@title}
\newcommand{\newreptheorem}[2]{%
\newenvironment{rep#1}[1]{%
 \def\rep@title{#2~\ref{##1}}%
 \begin{rep@theorem}}%
 {\end{rep@theorem}}}
\makeatother

\definecolor{RedOrange}{cmyk} {0, 0.77, 0.87, 0}
\definecolor{RoyalPurple}{cmyk} {0.84, 0.53, 0, 0}
\definecolor{YellowGreen}{cmyk} {0.44, 0, 0.74, 0}
\definecolor{Fuchsia}{cmyk} {0.47, 0.91, 0, 0.08}
\definecolor{Blue}{cmyk} {0.84, 0.53, 0, 0}
\definecolor{BlueViolet}{cmyk} {0.84, 0.53, 0, 0}
\definecolor{Black}{cmyk} {0.75, 0.68, 0.67, 0.9}
\usepackage{xcolor}
\usepackage{verbatim}
\usepackage{amsmath}
\usepackage{amssymb, mathrsfs}
\usepackage{amsbsy}

\usepackage{amscd}
\usepackage{amsthm}

\usepackage[english]{babel}
\usepackage{todonotes}
\usepackage[T1]{fontenc}

\usepackage{color}

\newcommand{\R}{\mathbb{R}}

\newcommand{\G}{\mathcal{G}}
\newcommand{\N}{\mathbb{N}}
\newcommand{\e}{\varepsilon}

\newcommand{\E}{\mathbb{E}}
\newcommand{\Z}{\mathbb{Z}}

\renewcommand{\P}{\mathbb{P}}

\newcommand{\eps}{\varepsilon}

\newcommand{\kF}{\mathcal{F}}

\newcommand{\dd}{{\rm d}}

\newcommand{\1}[1]{\mathbf{1}_{#1}}

\newcommand{\lin}{\left[\kern-0.15em\left[}
\newcommand{\rin} {\right]\kern-0.15em\right]}
\newcommand{\linf}{[\kern-0.15em [}
\newcommand{\rinf} {]\kern-0.15em ]}
\newcommand{\ilin}{\left]\kern-0.15em\left]}
\newcommand{\irin} {\right[\kern-0.15em\right[}

\def\al#1{\begin{align*}#1\end{align*}}
\def\aln#1{\begin{align}#1\end{align}}

\usepackage{constants}

\newconstantfamily{c}{symbol=c}

\newconstantfamily{a}{symbol=\alpha}

\newcommand{\secno}[1]{\thesection.\arabic{#1}}
\newconstantfamily{kE}{
symbol=\mathcal{E},
format=\secno,
reset={section}
}

\renewcommand{\hat}{\widehat}
\renewcommand{\tilde}{\widetilde}

\newcommand{\cev}[1]{\accentset{\leftarrow}{#1}}
\newcommand{\p}{{p^*}}

\newtheorem{lem}{Lemma}[section]
\newtheorem{remark}[lem]{Remark}

\newtheorem{prop}[lem]{Proposition}
\newtheorem{thm}[lem]{Theorem}

\newtheorem{cor}[lem]{Corollary}

\newtheorem{thmx}{Theorem}

\newtheorem {rem}[lem] {Remark}

\newcounter{assu}
\usepackage{color}
\definecolor{lilas}{RGB}{182, 102, 210}

\newcommand{\CO}{}
\numberwithin{equation}{section}

\renewcommand{\S}{S}

\newcommand{\K}{K}

\newcommand{\M}{M}

\newcommand{\F}{\mathcal F}

\newcommand{\Dprev}{\Delta^{\operatorname{prev}}}
\newcommand{\Dmg}{\Delta^{\operatorname{mg}}}
\newcommand{\s}{{s^*}}

\title[Equivalence of fluctuations of discretized SHE and KPZ in subcritical weak disorder]{Equivalence of Fluctuations of discretized SHE and KPZ equations in the Subcritical Weak Disorder Regime}
\date{\today}

\author{Shuta Nakajima} 
\address[Shuta Nakajima]
{Meiji University, Kanagawa, Japan}
\email{njima(at)meiji.ac.jp}

\author{Stefan Junk} 
\address[Stefan Junk]
{Gakushuin University, 1-5-1 Mejiro, Toshima-ku, Tokyo 171-8588, Japan}
\email{sjunk(at)math.gakushuin.ac.jp}

\begin{document}

\maketitle

\begin{abstract}
We study the fluctuations of discretized versions of the stochastic heat equation (SHE) and the Kardar-Parisi-Zhang (KPZ) equation in spatial dimensions $d\geq 3$ in the weak disorder regime. The discretization is defined using the directed polymer model. Previous research has identified the scaling limit of both equations under a suboptimal moment condition and, in particular, it was established that both converge { in law} to the same limit. We extend this result by showing that the fluctuations of both equations are close in probability in the subcritical weak disorder regime, indicating that they share the same scaling limit (the existence of which remains open). Our result applies under a moment condition that is expected to hold throughout the interior of the weak disorder phase, which is currently only known under a technical assumption on the environment. We also prove a lower tail concentration of the partition functions.
\end{abstract}

\section{Introduction}

\subsection{Motivation}\label{sec:motivation}
The KPZ equation is formally defined as 
\begin{align}\label{eq:KPZ}
\partial_t h = \frac{1}{2} \Delta h - \gamma |\nabla h|^2 + \dot{W},
\end{align}
where \( h = h(t, x) \) with \( t \in \mathbb{R} \) and \( x \in \mathbb{R}^d \). Here, \(\Delta\) and \(\nabla\) represent the Laplacian and the gradient, respectively, and \(\dot{W}\) denotes (uncorrelated) space-time white noise. Since its introduction in the seminal paper \cite{PhysRevLett.56.889}, where it was proposed as a prototype for the scaling limit of a certain class of interface growth models, the KPZ equation has received significant attention from both the physical and mathematical communities {due to its conjectured universality (see the review article \cite{HT15}}).  

\smallskip Note that \eqref{eq:KPZ} is, at least formally, related to the multiplicative stochastic heat equation
\begin{align}\label{eq:SHE}
\partial_t u=\frac 12\Delta u + \gamma u\cdot \dot{W}
\end{align}
through a substitution $h=\gamma^{-1}\log u$, which is known as the Cole-Hopf transformation.

\smallskip Much effort has been devoted to making rigorous sense of the solutions {to both equations}, and it is now widely accepted that the one-dimensional case is well-understood. To see why \( d=1 \) is special, note that the term \( u \cdot \dot{W} \) in \eqref{eq:SHE} makes sense in \( d=1 \) as an Itô-integration, and thus the solution to \eqref{eq:SHE} is well-defined \cite{bertini1997stochastic}. For dimensions $d \geq 2$, however, the KPZ equation becomes more singular due to the higher dimensionality of the space-time white noise and the nonlinearity $|\nabla h|^2$ leads to considerable difficulties. Moreover, the roughness of the white noise implies that any solution is, at best, a distribution, and the square of a distribution $|\nabla h|^2$, as appears in the second term in \eqref{eq:KPZ}, is not well-defined. 

\smallskip One natural approach to deal with this problem is to discretize \eqref{eq:KPZ} and \eqref{eq:SHE} {as follows}. Let $(X(k,x))_{k\in\N,x\in \Z^d}$ be independent and identically distributed (i.i.d.)  centered random variables. Fix $n\in \N$ and consider the following difference equation: for $k=0,\dots,n$ and $x\in\Z^d$, 
\begin{equation}\label{eq: discrete-SHE}
\begin{split}
    U_n(0,x)&=1,\\
(\partial^{\operatorname{disc}} U_n(\cdot,x))(k)&=(\Delta^{\operatorname{disc}} U_n(k,\cdot))(x)+\frac 1{2d}\sum_{y\sim x}  U_n(k,y) X(k,y),
\end{split}
\end{equation}
{where $(\partial^{\operatorname{disc}} f)(k)\coloneqq f(k+1)-f(k)$ is the discrete time derivative and} $(\Delta^{\operatorname{disc}}f)(x)\coloneqq \frac 1{2d}\sum_{y\sim x}(f(y)-f(x))$ is the discrete Laplace operator and the sum is over all neighbors of $x$. This difference equation closely resembles \eqref{eq:SHE}, except that an average is taken after the application of the multiplication operator. The solution to this equation certainly exists and has been studied as directed polymer models (see Remark~\ref{remark: polymer-SHE} below). We further define $H_n(k,x) \coloneqq  \log U_n(k,x)$ following the prescription of the Cole-Hopf transformation, which {satisfies} a discretized version of \eqref{eq:KPZ} (see the discussion before Theorem~\ref{thmx:K_n}). Note that this discretization approach is similar in spirit to the point of view of the original paper \cite{PhysRevLett.56.889}, where the KPZ equation is related to the scaling limit of discrete growth models.

\smallskip In this paper, we are interested in the relation between the two equations  \eqref{eq:KPZ} and \eqref{eq:SHE} in higher dimensions, that is, for $d\geq 3$, where \eqref{eq:SHE} is  ill-posed. Similar to the role of $\gamma$, {the discretization as introduced in Section \ref{sec:polymer}} has a parameter $\beta$ controlling the strength of the noise (i.e., $X$ depends on the parameter $\beta$) and we are interested in the so-called \emph{weak disorder} regime, i.e., the set of $\beta$ where the influence of the noise becomes asymptotically weak and both $U_n$ and $H_n$ are expected to converge, after suitable re-centering and re-scaling, to a limit: the stochastic heat equation with {additive} white noise, or Lévy noise. 

\smallskip In previous work, the scaling exponent $\xi=\xi(\beta)$ of $U_n$ has been identified in the whole weak disorder regime \cite{J22}, and the scaling limits of both   $U_n$ and $H_n$ were shown to be the Edwards-Wilkinson equation, {i.e., the SHE with additive white noise,} for $\beta$ up to a critical value $\beta_2$ \cite{comets2017rate,CN21,lygkonis2022edwards}, which is known to differ from the critical value $\beta_c$ for the weak disorder regime.   The main result of this paper is to prove that the difference between $U_n$ and $H_n$ is much smaller than $n^{-\xi}$ with high probability whenever $\beta<\beta_c$. 

\smallskip There are two main consequences of this result: Firstly, {it proves that} the scaling exponent of $H_n$ is the same as that of $U_n$ when $\beta<\beta_c$. Secondly, if future work identifies the correct scaling limit for $U_n$ (we conjecture it to be the Edwards-Wilkinson equation with Lévy noise), then $H_n$ will converge to the same limit. This is significant because $U_n$, being a discretization of a linear equation, is conceptually easier to analyze than $H_n$.

\smallskip 
{Our main result is thus to show that $U_n$ and $H_n$ are close not only in law but also in probability. A similar statement first appeared in \cite[Section 1.3]{dunlap2020fluctuations} for sufficiently small $\beta$ in dimension $d\geq 3$ and is also implicit in the proofs of \cite{lygkonis2022edwards,cosco2022law} in the $L^2$-regime. Moreover, for $d=2$ the same is known in the subcritical regime, see \cite[Section~2.1]{CSZ20} and \cite{nakajima2023fluctuations}. All of these results are based on $L^2$-computations and the novelty of our work is thus to formalize the proofs and to extend the validity beyond the $L^2$-regime.} Our results indicate that the discretization of the nonlinear SPDE \eqref{eq:KPZ} effectively behaves linearly in the scaling limit, a phenomenon that opens new avenues for analyzing such equations in higher dimensions. We believe that our techniques, with some modifications, can also be applied to a {continuous space-time setting} studied in \cite{magnen2018scaling,gu2018edwards,cosco2022law}.

\subsection{The directed polymer model}\label{sec:polymer}
Let us introduce the directed polymer model, which is used to discretize the equations \eqref{eq:KPZ} and \eqref{eq:SHE}. This model was first introduced in \cite{huse1985pinning} to describe the interfaces in the two-dimensional Ising model with random coupling. Beyond its relation to the KPZ equation, its mathematical properties are remarkably interesting by themselves, and the model has been the focus of research, beginning with the seminal works \cite{imbrie1988diffusion,bolthausen1989note}. We refer to \cite{C17,zygouras2024directedpolymersrandomenvironment} for recent review articles.

Let $(\omega_{n,x})_{n\in\N,x\in\Z^d}$ be i.i.d.\ random variables with law $\P$ satisfying
\begin{align}\label{eq:expmom}
\mathbb{E}[e^{\beta|\omega_{1,0}|}]<\infty\qquad\text{ for all }\beta\geq 0.
\end{align} 
 The main interest in this context is to understand the long-term behavior of the \emph{polymer measure}, defined for $\beta\geq 0$ and $n\in\N$ as 
\begin{align*}
\mu_{\omega,n}^\beta(A)=(Z_{\omega,n}^\beta)^{-1}E[e^{\beta\sum_{k=1}^n\omega_{k,X_k}-n\lambda(\beta)}\mathbf{1}_A] .
\end{align*}
Here $P$ and $E$ denote the law and its expectation of the simple random walk $(X_k)_{k\geq 0}$ starting at $0$,  $Z_{\omega,n}^\beta$ is the normalization constant, called \emph{partition function} in this context:
\aln{
Z_n\coloneqq  Z_{\omega,n}^\beta\coloneqq E\left[e^{\beta\sum_{k=1}^n\omega_{k,X_k}-n\lambda(\beta)}\right],
}
and $\lambda(\beta)$ denotes the cumulative generating moment function of $ \omega_{n,x}$, i.e., $\lambda(\beta)\coloneqq \log{\mathbb{E}[e^{\beta \omega_{n,x}}]}$.  Let  { $\theta_{n,x}$} denote the space-time shift acting on the environment, i.e., $(\theta_{n,x}\omega)_{m,y}\coloneqq \omega_{n+m,x+y}$. For a random variable $Y$ with respect to $\omega$, we define $Y\circ \theta_{n,x} (\omega) \coloneqq  Y(\theta_{n,x} \omega)$. Note that $Z_n\circ \theta_{k,x}$ is the partition function starting at position $x$ and at time $k$ {and time-horizon $\llbracket k+1,n+k\rrbracket$}.
\begin{remark} \label{remark: polymer-SHE}
 For $n\in\N$, $k=0,\dots,n$ and $x\in\Z^d$, since $$Z_{n-(k+1)}\circ\theta_{k+1,x}=\frac 1{2d}\sum_{y\sim x}e^{\beta\omega_{n-k,y}-\lambda(\beta)}Z_{n-k}\circ\theta_{k,x},$$ 
 we obtain that $U_n(k,x)\coloneqq Z_{n-k}\circ\theta_{k,x}$ solves \eqref{eq: discrete-SHE} with $X(k,x)=e^{\beta\omega_{n-k,x}-\lambda(\beta)}-1$.
\end{remark}
A crucial observation is that $(Z_n)_{n\in\N}$ is a non-negative martingale with respect to the filtration $\F_n\coloneqq \sigma(\omega_{k,x}:k\leq n,\,x\in \Z^d)$ and thus converges to an almost sure limit $Z_\infty\coloneqq \lim_{n\to\infty}Z_n$. It is easy to see that $Z_\infty$ satisfies a zero-one law, $\mathbb{P}(Z_\infty>0)\in\{0,1\}$. We say that \emph{weak disorder}, resp. \emph{strong disorder}, holds if $\mathbb{P}(Z_\infty>0)=1$, resp. $\mathbb{P}(Z_\infty=0)=1$. It was shown \cite{CY06} that a phase transition occurs between these regimes, i.e., there exists $\beta_c=\beta_c(d)\in[0,\infty{]}$ such that weak disorder holds for $\beta\in[0,\beta_c)$ and strong disorder holds for $\beta>\beta_c$. Moreover, $\beta_c=0$ in $d=1$ and $d=2$, and $\beta_c>0$ in $d\geq 3$. Recently, it was further shown {\cite{junk2024strongdisorderstrongdisorder, JL25_1} that weak disorder holds at $\beta_c$ as well.} Much information has been obtained about the behavior of this model in the weak and strong disorder regimes and we refer to the surveys mentioned above for details. For our purposes, it is important to introduce the \emph{critical exponent} $\p(\beta)$:
$$\p(\beta)\coloneqq \sup\Big\{p\geq 1\colon \sup_{n\geq 1}\mathbb{E}[Z_n^p]<\infty\Big\},$$
and the critical temperature for $L^p$-boundedness,
\begin{align*}
\beta_p\coloneqq \sup\Big\{\beta\geq 0\colon\p(\beta)>p\Big\}.
\end{align*}
{It is clear from the definition that $\beta_q\leq \beta_p\leq \beta_c$ for any $1<p\leq q$.} Let us summarize {some} known properties of this exponent:
\begin{thmx}\label{thmx:p}
\begin{enumerate}
    \item[(i)] Weak disorder implies that $\p(\beta)\geq 1+\frac 2d$.
    \item[(ii)] In $d\geq 3$, it holds that $\beta_2>0$. If $\beta_2<\infty$, then $\beta_2<\beta_c$.
    \item[(iii)] If $\p(\beta)>1$, then $\p$ is left-continuous at $\beta$. Moreover, if $\p(\beta)\in(1+\frac 2d,2]$, then $\p$ is also right-continuous at $\beta$.
    \item[(iv)] {It holds that $\p(\beta_c)=1+\frac 2d$}.
    \item[(v)] If $\beta_2<\infty$, then $\beta_{1+2/d}>\beta_2$. Moreover, $\beta_{1+2/d}=\beta_c$ if $\omega$ is finitely supported, i.e., 
    \begin{align}\label{eq:finitely}
    \text{there exists $A\subseteq\R$ such that }|A|<\infty\text{ and }\mathbb{P}(\omega_{1,0}\in A)=1.
    \end{align}
\end{enumerate}
\end{thmx}

Parts (i) and (iii) have recently been proved in \cite[Corollaries 2.7 and 2.9]{junk2024tail} and \cite[Theorem 1.2(i)]{junk2023local}. The observation that $\beta_2>0$ in $d\geq 3$ goes back to \cite{bolthausen1989note} and the strict inequality $\beta_2<\beta_c$ was first proposed in \cite{B04} and then proved over a number of works, see \cite[Remark 5.2]{C17} for the precise references. Parts (iv) and (v) are proved in \cite[Corollary 2.2]{junk2024strongdisorderstrongdisorder} and \cite[Theorem 1.2(ii)]{junk2023local}.

\smallskip The results of this paper are valid for $\beta<\beta_{1+2/d}$, i.e., under the assumption, 
\begin{align}\tag{SWD}\label{eq:assumption} 
\p(\beta)>1+\frac{2}{d},
\end{align}
which we will call \emph{subcritical weak disorder} in analogy with the terminology in the two-dimensional case. {Note that since $\beta_c=0$ for $d=1$ and $d=2${, see \cite[Theorem~1.3]{CSY03},} \eqref{eq:assumption} implies $d\geq 3$.}

\smallskip {Note that by Theorem~\ref{thmx:p}(iv), $\p(\beta)>1+2/d$ is known to hold for some range of parameters beyond $\beta_2$ without any additional assumptions on the environment, so that the subcritical weak disorder regime is strictly larger than the $L^2$-regime{, i.e.,} $[0,\beta_2)$.} While $\beta_{1+2/d}=\beta_c$ is currently only known under the rather restrictive assumption \eqref{eq:finitely}, we believe it to be true more generally and that the main result below is thus valid in the whole interior $[0,\beta_c)$ of the weak disorder phase.  In the proofs, we do not use \eqref{eq:finitely}, only \eqref{eq:assumption}. 

\smallskip {The reason why we require \eqref{eq:assumption} is that we crucially use} the local limit theorem proved in \cite{junk2023local}. To state it, let us introduce the notation $(m,x)\leftrightarrow(n,y)\iff P(X_{n-m}=y-x)>0$ to mean that $(m,x)$ and $(n,y)$ have the same parity, and for the $(m,x)\leftrightarrow(n,y)$ the \emph{pinned partition function}
\begin{align*}
Z_{(m,n)}^{m,x;n,y}\coloneqq E^{m,x;n,y}[e^{\sum_{i=m+1}^{n-1} \omega_{i,X_i}-(n-m-2)\lambda(\beta)}],
\end{align*}
where $E^{m,x;n,y}$ denotes the expectation with respect to the random walk bridge between $(m,x)$ and $(n,y)$. {Note that neither the environment at $(m,x)$ nor at $(n,y)$ is taken into account}. The local limit theorem essentially states that, for a suitable {range} of endpoints, the pinned partition function $Z_{(0,n)}^{0,0;n,y}$ has the same integrability as $Z_n$.
\begin{thmx}[{\cite[Theorem 1.1(i)]{junk2023local}}]\label{thmx:local}
{Assume \eqref{eq:assumption} and let $p\in(1+2/d,\p(\beta)\wedge 2)$. There exist $r=r(\beta,p)$ and $C=C(\beta,p)>0$ such that, for all $n\in\N$, $y\in\Z^d$ with $|y|\leq rn$ and $(0,0)\leftrightarrow(n,y)$,
\begin{align*}
\mathbb{E}[(Z_{(0,n)}^{0,0;n,y})^p]\leq C.
\end{align*}}
\end{thmx}
Finally, we need to make a mild technical assumption that ensures that the inverse partition function $Z_n^{-1}$ is well-behaved in the weak disorder phase. More precisely, we assume that the environment $\P$ satisfies the following concentration property:  
\begin{equation}\tag{CONC}\label{eq:concprop}
\begin{split}
    &\exists C>0, \exists \gamma\in (1,2],\,\forall m\in\N\text{ and }\forall f:\R^m\to\R\text{ convex and $1$-Lipschitz}:\\
    &\qquad\mathbb{P}(|f(\omega_m)-\mathbb{E}[f(\omega_m)]|>t)\leq Ce^{-t^\gamma/C },
\end{split}
\end{equation}
where $\omega_m$ denotes the first $m$ coordinates of {$(\omega_{n,x})_{n\in\N,x\in\Z^d}$} in an arbitrary enumeration of $\N\times\Z^d$.  Assumption \eqref{eq:concprop} is satisfied if $\omega$ is bounded or Gaussian, or more generally if $\P$ satisfies a log-Sobolev inequality, and goes back to a result from \cite{L01} (see the discussion in \cite[Section 2.1]{CTT17}). {Property \eqref{eq:concprop} has already played an important result in the analysis of the \emph{intermediate disorder regime} in dimension $2$, see \cite{CSZ20}, and we extend these techniques to $d\geq 3$.} The Poisson distribution is an example for an environment such that {\eqref{eq:expmom} holds} but \eqref{eq:concprop} fails. The significance of \eqref{eq:concprop} is that it implies the lower tail concentration of the partition function. {We consider $Z_n^{x} \coloneqq  Z_n\circ \theta_{0,x}$ the partition function starting at position $x$ and at time $0$.}

\begin{thm}\label{thm: lower tail concentration}
Assume that \eqref{eq:concprop} is satisfied. If $\p(\beta)>1+2/d$, then there exists $C>0$ such that, for all $u\geq 1$ and $n\in\N$,
\begin{align}\label{eq:newtail}
    \mathbb{P}(Z_n\leq 1/u)\leq Ce^{-(\log u)^\gamma/C}.
\end{align}
    In particular, $Z_{\infty}$ has all negative moments and the following uniform bound holds: for all $\eps>0$ and $k\in\N$ there exists $C>0$ such that, for all $A\subseteq\Z^d$,
    \begin{align}\label{eq:union}
    \E\Big[\max_{x\in A,n\in\N}(Z_n^x)^{-k}\Big]\leq C|A|^\eps.
    \end{align}
\end{thm}

The conclusion of Theorem~\ref{thm: lower tail concentration} was previously known in $L^2$-regime  \cite{ben2009large}, i.e., under the assumption $\beta<\beta_2$. Our proof follows the same argument and relies on Theorem~\ref{thmx:local} for the key estimate.\\

\subsection{Relation between directed polymer model and KPZ/SHE}

 \smallskip To describe the scaling limit of {$U_n$ and $K_n$ introduced in Section \ref{sec:motivation}}, we denote the \emph{scaling exponent} as
\begin{align}\label{eq:xi}
 \xi\coloneqq \xi(\beta)\coloneqq \frac{d}{2}-\frac{1+\frac{d}{2}}{\p(\beta)\wedge 2}.
\end{align}
{Note that Theorem~\ref{thmx:p}(i) implies $\xi(\beta)\geq 0$ and that $\xi(\beta)>0$ if and only if \eqref{eq:assumption} holds.}

\smallskip  {Let $\mathcal C_c(\R^d)$ denote the set of continuous, compactly supported functions $f\colon\R^d\to\R$.}
\begin{thmx}\label{thmx:EW}
{For $f\in\mathcal C_c(\R^d)$}, set 
\begin{align}\label{eq:def_Sn}
S_n(f)\coloneqq n^{-d/2}\sum_{x\in\Z^d}f(x/\sqrt n) (Z_n^{x}-1).
\end{align}
\begin{enumerate}
    \item[(i)]In weak disorder, $\lim_{n\to\infty}S_n=0$ in probability.
    \item[(ii)]In $L^2$-regime, i.e., for $\beta<\beta_2$, $n^{\xi} S_n(f)$ converges in law as $n\to\infty$ to a centered Gaussian with an explicit variance given in \cite[Theorem 1.2]{lygkonis2022edwards}.
    \item[(iii)]Assume $f\not\equiv 0$. In weak disorder, for any $\eps>0$ it holds that
    \begin{align*}
        \lim_{n\to\infty}\P\Big(n^{-\xi-\eps}\leq |S_n(f)|\leq n^{-\xi+\eps}\Big)=1.
    \end{align*}
\end{enumerate}
\end{thmx}
Part (i) of the previous result {is} a law of large numbers for $U_n$ \cite[Theorem C(i)]{J22} if the starting point is averaged with a suitable test function and part (ii) shows that the scaling limit for the fluctuations after the appropriate centering is the Edwards-Wilkinson equation with Gaussian white noise \cite{lygkonis2022edwards}. Finally, part (iii) identifies the correct scaling exponent in the whole of the weak disorder regime \cite{J22} ({with a sub-polynomial rate of decay at $\beta_c$}). 

\smallskip {Let us also mention that the variance of the limiting object in part (ii) diverges as $\beta\uparrow\beta_2$, so the regime $\beta>\beta_2$ requires different techniques.} We conjecture  that the scaling limit of $S_n(f)$ for $\beta\in(\beta_2,\beta_c)$ should be the Edwards-Wilkinson equation with $\frac{\p(\beta)}2$-Lévy noise, instead of white noise, but {at the moment we do not know} how to make this rigorous. {Note also} that Theorem~\ref{thmx:EW}(ii) does not cover the value $\beta_2$ itself, while on the other hand part (iii) together with Theorem~\ref{thmx:p}(iii) show that {$\p(\beta_2)=2$} and  $|S_n(f)|=n^{-\frac{d-2}4+o(1)}$, i.e., the same scaling exponent as in the case $\beta<\beta_2$ appears. We believe that the scaling limit in this case should be the same as for $\beta<\beta_2$ after an appropriate logarithmic correction has to be{en} added in the scaling.

\smallskip Let us now {discuss} the discretization of the KPZ equation \eqref{eq:KPZ}. As was first proposed by \cite{bertini1997stochastic}, the ``correct'' way to obtain a solution to the KPZ equation is via the Cole-Hopf transformation, i.e., by taking the logarithm of a solution to the stochastic heat equation. Hence, we  define $H_n(k,x)\coloneqq \log U_n(k,x)$.  Recalling the discrete time derivative $\partial^{\operatorname{disc}}$ and discrete Laplacian $\Delta^{\operatorname{disc}}$ defined in \eqref{eq: discrete-SHE}, we additionally introduce a discretization of the non-linearity, $|\nabla^{\operatorname{disc}} f|^2(x)\coloneqq \sum_{y\sim x}(f(y)-f(x))^2$. Then
\begin{align*}
(\partial^{\operatorname{disc}}H_n(\cdot,x))(k)&=\log\Big(1+\frac{(\partial^{\operatorname{disc}}U_n(\cdot,x))(k)}{U_n(k,x)}\Big)\approx \frac{(\partial^{\operatorname{disc}}U_n(\cdot,x))(k)}{U_n(k,x)},\\
(\Delta^{\operatorname{disc}} H_n(k,\cdot))(x)&=\frac{1}{2d} \sum_{y\sim x} \log\Big(1+\frac{U_n(k,y)-U_n(k,x)}{U_n(k,x)}\Big)\approx \frac{(\Delta^{\operatorname{disc}}U_n(k,\cdot))(x)}{U_n(k,x)}{-}\frac 1{4d} \frac{|\nabla^{\operatorname{disc}}U_n(k,\cdot)|^2(x)}{U_n(k,x)^2},\\
|\nabla^{\operatorname{disc}} H_n(k,\cdot)|^2(x)&=\sum_{y\sim x}\log\Big(1+\frac{U_n(k,y)-U_n(k,x)}{U_n(k,x)}\Big)^2\approx \frac{|\nabla^{\operatorname{disc}}U_n(k,\cdot)|^2(x)}{U_n(k,x)^2},
\end{align*}
{where we have used the first-order Taylor approximation, $\log(1+u)\approx u$, in the first and third lines and the second-order Taylor approximation, $\log(1+u)\approx u- u^2/2$, in the second line. Together with \eqref{eq: discrete-SHE}, we obtain an equation that resembles \eqref{eq:KPZ} with an appropriate choice for $\gamma$.}

\smallskip {Note that in the above derivation, the use of Taylor's formula is not justified since the ``$u$'' in question is not actually small -- for example, $U_n(k,y)-U_n(k,x)$ is of order $O(1)$ even if $x$ and $y$ are close. Let us try to sketch how this might be circumvented. For $\delta\in(0,1)$ and $k\ll n$, we write $\tilde Z_n^{k,x}\coloneqq {\CO (Z_{n-k}\circ\theta_{k,x})/(Z_{n^{\delta}}\circ\theta_{k,x})}$ and decompose
\begin{align*}\label{justify}
(\partial^{\operatorname{disc}}H_n(\cdot,x))(k)=\log\left(1+\frac{\tilde Z_n^{k+1,x}-\tilde Z_n^{k,x}}{\tilde Z^{k,x}_n}\right)+\log Z_{n^{\delta}}\circ\theta_{k+1,x}-\log Z_{n^{\delta}}\circ\theta_{k,x}.
\end{align*}
Since $Z_n$ converges to a positive limit, we know that $\tilde Z_n^{k,x}$ is close to one (this is made precise in Proposition~\ref{prop:diff}) and thus Taylor's formula can be applied to the first term. On the other hand, the last two terms depend only on the environment close to $(k,x)$, and moreover they have the same law and a finite second moment. By the central limit theorem, we can expect that the contribution from these terms is of order $n^{-d/4}\ll n^{-\xi}$ when an average over the diffusive scale is taken, which is negligible. More work is required to make this argument rigorous -- in particular, it needs to be checked whether the application of Taylor's formula to $\tilde Z_n$ can be interpreted as a discretization of \eqref{eq:KPZ}. This goes beyond the scope of the current introduction and we leave it as an interesting question for future research. Let us mention, however, that the idea outlined here is the basis of our proof, see also the discussion after Corollary \ref{corrr}.}

\smallskip The following theorem summarizes the known results about the long-term behavior for the discretized KPZ equation:
\begin{thmx}\label{thmx:K_n}
For $f\in\mathcal C_c(\R^d)$, set 
\begin{equation}
K_n(f)\coloneqq n^{-d/2}\sum_{x\in\Z^d}f(x/\sqrt n) (\log Z_n^{x}-\mathbb{E}[\log Z_n]).\label{eq:def_Kn}
\end{equation}
\begin{enumerate}
    \item[(i)]In weak disorder, $\lim_{n\to\infty}K_n=0$ in probability.
    \item[(ii)]In $L^2$-regime, i.e., for $\beta<\beta_2$, $n^{\xi} K_n(f)$ converges in law as $n\to\infty$ to a centered Gaussian with {the same variance as in Theorem~\ref{thmx:EW}(ii)}.
\end{enumerate}
\end{thmx}
Note that by Theorem~\ref{thm: lower tail concentration}, $\mathbb{E}[\log Z_n]$ converges to $\mathbb{E}[\log Z_\infty]>-\infty$ under the assumption \eqref{eq:concprop}. As far as we know, part (i) of the previous theorem is not explicitly stated in the literature, but it is clear that it can be proved by the same argument as in \cite[Theorem C(i)]{J22}. Part (ii) is proved in \cite{lygkonis2022edwards} {and in \cite{cosco2022law} in a related, continuous setting}. Note that a statement equivalent to Theorem~\ref{thmx:EW}(iii) is missing and will be proved as part of our main result. Finally, we point out that the random variables in Theorem~\ref{thmx:EW}(ii) and~\ref{thmx:K_n}(ii) converge to the same limits in law  -- as mentioned earlier, one interesting consequence of our result is that the two random variables are also close \emph{in {probability}}.

\subsection{The main result}
Assume $d\geq 3$. We are interested in the fluctuations $S_n(f)$ and $K_n(f)$ with initial condition $f\in\mathcal C_c(\R^d)$ defined in \eqref{eq:def_Sn} and \eqref{eq:def_Kn}. 
 
\smallskip The following is our main result:
\begin{thm}\label{thm:main}
Assume \eqref{eq:assumption} and \eqref{eq:concprop}. For every $f\in\mathcal C_c(\R^d)$ there exists $\eps{=\eps(\beta)}>0$ such that,
    \begin{align}\label{eq:main}
	\lim_{n\to\infty}\mathbb{P}\big(|\S_n(f)-\K_n(f)|\geq n^{-\xi-\eps}\big)=0.
\end{align}
\end{thm}

\begin{rem}
  The value of $\eps(\beta)$ in \eqref{eq:main} could, in principle, be made explicit. Since there is no reason to believe that our methods yield the best-possible bound,  we did not optimize our argument and instead focused on keeping the proof concise.
\end{rem}
\begin{remark}
{\CO Part of our argument, namely the approximation of $K_n(f)$, extends beyond the logarithmic case and covers general nonlinear transformations $F(u)$ satisfying the conditions $F(1)=1$, $F'(u)>0$, and $F''(u)<0$, similar to the analysis in the $L^2$-phase in \cite{dunlap2020fluctuations, nakajima2023fluctuations}. More precisely, from the derivation in \cite[Section~3.1]{nakajima2023fluctuations} we expect that $F(Z_n^x)-\E[F(Z_n)]$ can be approximated by 
\begin{align*}
\tilde M_n^\delta(x)=\sum_{k\in\llbracket n^{1-\delta}+1,n\rrbracket} F'(Z_n^x)Z_{\ell_n}^x \cev Z^{k,y}_{[k-\ell_n,k)} E_{k,y}p_k(x,y),
\end{align*}
where $1\ll\ell_n\ll n$ is some intermediate scale. In the logarithmic case $F(u)=\log u$, the term $F'(Z_n^x)Z_{\ell_n}^x$ 
converges to one. Hence, $\tilde M_n^\delta(x)$ asymptotically agrees with the corresponding term $M_n^\delta(x)$ in the approximation of $S_n(f)$, see \eqref{eq:M_n} below. In the general case, we obtain that 
\begin{align*}
\tilde K_n(f):=n^{-d/2}\sum_xf(x/\sqrt n) \big(F(Z_n^x)-\E[F(Z_n)]\big)
\end{align*} is approximated by $n^{-d/2}\sum_x f(x/\sqrt n) \tilde M_n^\delta(x)$, but beyond the $L^2$-phase the fluctuations of that expression have been computed only in the logarithmic case, see Theorem~\ref{thmx:EW}(iii). Thus, further research is needed to generalize the result in \cite{J22} in order to determine the fluctuations of $\tilde K_n(f)$ with our methods.}
\end{remark}

From Theorem~\ref{thmx:EW}(iii), we conclude the following:
\begin{cor}\label{corrr}
Assume \eqref{eq:assumption} and \eqref{eq:concprop}. If $f\not\equiv0$, then for any $\eps>0$,
\al{
\lim_{n\to\infty}\mathbb{P}\big(n^{-\xi-\eps}\leq |\K_n(f)|\leq n^{-\xi+\eps}\big)= 1.
}
Moreover, there exists $\eps_0>0$ such that,
$$\lim_{n\to\infty}\mathbb{P}\left(\frac{|\S_n(f)-\K_n(f)|}{\min\{|\S_n(f)|,|\K_n(f)|\}}\leq n^{-\eps_0}\right)=1.$$
\end{cor}
We also note that our result implies that $K_n$ has the same scaling limit as $S_n$, even though the existence of such a scaling limit is not known at the moment for $\beta\geq \beta_2$.
\begin{cor}
Under the assumptions of Theorem~\ref{thm:main}, assume that $(n_k)_{k\in\N}$ and $(a_k)_{k\in\N}$ are such that $a_kS_{n_k}(f)$ converges in distribution as $k\to\infty$. Then $a_kK_{n_k}(f)$ converges to the same limit in distribution.
\end{cor}

{Let us explain why this} corollary is surprising: Note that, since the second moment of $Z_n^\beta$ diverges exponentially fast for $\beta>\beta_2$, {one naturally expects} that the scaling limit of $S_n(f)$ {is non-}Gaussian, and a close examination of the proof of \cite[Theorem 1.1]{J22} suggests that the scaling limit of $S_n(f)$ should be a stable random variable in the regime $\beta\in (\beta_2,\beta_c)$. On the other hand, one would naturally conjecture that a scaling limit of $K_n(f)$ is Gaussian for the following reason: Observe that \( K_n(f) \) is a sum of identically distributed terms \( K_n(x) \) possessing all moments (see Proposition~\ref{prop: log correlation}). Moreover, \( K_n(x) \) and \( K_n(y) \) are asymptotically independent when the distance between \( x \) and \( y \) is sufficiently large. More precisely, since $K_n(x)$ converges almost surely for fixed $x$, we get $K_n(x)\approx K_{\ell_n}(x)$ for any $\ell_n\to \infty$. On the other hand, $K_{\ell_n}(x)$ and $K_{\ell_n}(y)$ are independent if $|x-y|_{\infty}>2\ell_n$. Our intuition based on the central limit theorem thus suggests that \( K_n(f) \) converges to a Gaussian limit in the whole weak disorder regime.

\smallskip Of course, this reasoning is not rigorous because one would need the approximation $K_n(x)\approx K_{\ell_n}(x)$ to hold uniformly in $x\in[-n^{1/2},n^{1/2}]^d$, and this is not true. In fact, we will find that, with $\ell_n=n^{1-o(1)}$, the contribution from averaging $K_{\ell_n}(x)$ is negligible compared to that of the approximation error $K_n(x)-K_{\ell_n}(x)$ (see Proposition~\ref{prop: key claims}), which leads to a decay rate $K_n(f)\approx n^{-\xi(\beta)+o(1)}$ strictly slower than what is observed in the Gaussian regime in Theorem~\ref{thmx:K_n}(ii).

\smallskip {We finish the discussion of our result by commenting on the difficulties arising from going beyond the $L^2$-regime from Theorems~\ref{thmx:EW}(ii) and \ref{thmx:K_n}(ii). The approach from \cite{lygkonis2022edwards} relies on an $L^2$-decomposition of the partition function, which does not converge for $\beta_2$. Moreover,} in the \( L^2 \)-regime for {a related, continuous space-time model} \cite{cosco2022law}, the proof uses an $L^2$-approximation of solutions from the  SHE to the KPZ equation, though this is not explicitly stated in their results. The main difficulty in extending their approach  beyond the \( L^2 \)-regime is that the covariance computation cannot be applied. Instead, we employ Burkholder's inequality (Lemma~\ref{lem:bh}) in several key steps{, which can be considered a conditional version of the $L^2$-computations and allows us to perform $L^p$-calculations, with suitable $p<2$.}

\subsection{Discussion of the SHE and KPZ equation}

Let us comment a bit more on the current knowledge about equations \eqref{eq:KPZ} and \eqref{eq:SHE}, and how our results {extend} the picture.

\smallskip First, we note that the success of the analysis of \eqref{eq:KPZ} and \eqref{eq:SHE} in \( d=1 \) is due to the breakthrough techniques, such as the theory of regularity structures \cite{hairer2014theory} introduced by Hairer and the paracontrolled calculus \cite{gubinelli2015paracontrolled} developed by Gubinelli, Imkeller, and Perkowski. These approaches have provided robust tools for handling the singularities that appear in the KPZ equation, making it possible to rigorously define and analyze its solutions to general stochastic partial differential equations (SPDE)  in low dimensions. It is worth noting that recent advances in integrable systems have enabled the derivation of explicit formulas for the one-dimensional KPZ equation \cite{sasamoto2010one,amir2011probability}. 

\smallskip Dimension $d=2$ is called critical while dimensions $d \geq 3$ are called supercritical. Here, the techniques mentioned before are ineffective and advanced renormalization techniques are needed to make sense of the solutions (see discussions in \cite{chatterjee2020constructing,magnen2018scaling}). In dimension two, early work \cite{chatterjee2020constructing} considered the renormalization with respect to the coefficient  \(\gamma\), i.e., the noise intensity \(\gamma = \gamma_\varepsilon\) scales as \(\gamma_\varepsilon = \hat{\gamma}/\sqrt{\log \varepsilon^{-1}}\), to properly handle the singularities arising from the regularization of the noise \(\dot{W}\), where the original noise is replaced by a smooth approximation \(\dot{W}^{\varepsilon}\) with parameter $\varepsilon>0$ (note that such reparametrizations were already considered in the one-dimensional KPZ equation, though with a different scaling, in \cite{AKQ14}). Notably, Caravenna, Sun, and Zygouras \cite{CSZ20}  showed that under certain small-noise conditions \(\hat{\gamma} < \hat{\gamma}_c\), referred to as the subcritical regime, the rescaled random fields associated with the SHE and KPZ equation converge to the Gaussian Free Field (GFF), corresponding to the Edwards-Wilkinson (EW) universality class---a linear stochastic partial differential equation that describes surface growth in the absence of nonlinear effects and serves as a baseline for understanding fluctuations in growth models (see also \cite{gu2020gaussian,dunlap2022forward,nakajima2023fluctuations,dunlap20232d} for related studies and its generalization). At criticality, i.e.,  $\hat{\gamma}=\hat{\gamma}_c$,  it was found that the  solutions to the SHE exhibit a universal random measure with logarithmic correlations  \cite{bertini1998two,caravenna2019moments,gu2021moments,caravenna2023critical}. {We also remark} that  a discrete analogue to  the SHE, {equivalent to the discretization introduced above}, is considered in \cite{CSZ20,caravenna2019moments,gu2021moments,caravenna2023critical,lygkonis2022edwards}. This limiting process, termed the { Critical 2D Stochastic Heat Flow}, provides a rigorous solution to the long-standing problem of defining solutions to the two-dimensional SHE  at criticality. However, despite these advances for the SHE, the problem of defining and analyzing solutions to the two-dimensional KPZ equation at criticality remains open and is a subject of ongoing research.

\smallskip For dimensions \( d \geq 3 \), the KPZ equation enters a supercritical regime in SPDE theory, where the solution theory becomes even more delicate. Early results focused on the the weak disorder regime, where the noise intensity is small enough for the system to exhibit diffusive behavior. Polymer representations and martingale methods were developed to describe the behavior of the SHE and KPZ  equation in {this regime}. Significant progress was made in understanding the law of large numbers and the role of the critical parameter \( \gamma \), where the system transitions from weak to strong disorder. In particular, the recent work \cite{mukherjee2016weak} introduced new tools from the theory of directed polymers to handle the higher-dimensional KPZ equation, where the noise intensity is scaled as $\gamma_\varepsilon = \varepsilon^{(d-2)/2} \hat{\gamma}$. Recent works \cite{magnen2018scaling,gu2018edwards,dunlap2020fluctuations,cosco2022law,lygkonis2022edwards} showed that certain solutions to the SHE and KPZ  equation with regularized noise converge to the solutions to Edwards-Wilkinson (EW) type equations, i.e., the stochastic heat equation with additive noise,  in certain ranges of noise intensities, called $L^2$-regime, where the regularized solutions possess uniform $L^2$ moments. This \(L^2\)-regime corresponds to the regime \(\beta < \beta_2\) discussed in the {previous section}. Research on the KPZ equation in \( d \geq 3 \) is still developing, with a significant focus on the \( L^2 \)-regime, {and} relatively few results are available outside the \( L^2 \)-regime. 

\subsection{Further notation}
Let $p_n(x)=P(X_n=x)$ denote the heat kernel of the simple random walk starting at $0$. Given $I\subseteq (0,\infty)$, the restricted Hamiltonian is defined by $H_I(\omega,X)\coloneqq \sum_{i\in I\cap \Z}\omega_{i,X_i}$, and we write $Z_I^{s,x}\coloneqq E^{s,x}[e^{\beta H_I(\omega,X)-|I\cap \Z|\lambda(\beta)}]$ for $I\subseteq (t,\infty)$, where $E^{s,x}$ is the expectation with respect to the simple random walk starting at $(s,x)$. We set $Z_I \coloneqq Z_I^{0,0}$.

\smallskip {In addition to }the \emph{point-to-plane} partition function $Z_n^{t,x}$ with arbitrary starting point, we also define the \emph{plane-to-point} partition function, or \emph{reverse partition function}, by $\cev Z_I^{t,y}\coloneqq \cev E^{t,y}[e^{\beta H_I(\omega,X)-|I\cap \Z|\lambda(\beta)}]$, where $I\subseteq (-\infty,t)$ and $\cev E^{t,x}$ is the expectation with respect to the time-reversed random walk with endpoint $(t,y)$,
\begin{align*}
\cev P^{t,y}((X_k)_{k=0,\dots,t}\in\cdot)\coloneqq P((X_k-X_t+y)_{k=0,\dots,t}\in \cdot).
\end{align*}
{Recalling also the \emph{point-to-point} partition function introduced before Theorem~\ref{thmx:local}}, we define $Z_{I}^{s,x;t,y}\coloneqq E^{s,x;t,y}[e^{\beta H_I(\omega,X)-|I\cap \Z|\lambda(\beta)}]$ for $I\subseteq [s,t]$. {The transition probability of the random walk bridge is denoted by} $p^{s,x;t,y}_k(z) \coloneqq  p_{k-s}(x)p_{t-k}(z-y)/p_{t-s}(y)=P^{s,x;t,y}(X_k=z)$.

\smallskip If $g$ is measurable with respect to  the filtration of the random walk $X$, we define 
\al{
Z_I^{s,x}[g]&\coloneqq E^{s,x}[e^{\beta H_I(\omega, X)-|I\cap \Z|\lambda(\beta)}g(X)],\\
\cev Z_I^{t,y}[g]&\coloneqq \cev E^{t,y}[e^{\beta H_I(\omega, X)-|I\cap \Z|\lambda(\beta)}g(X)],\\
Z_{I}^{s,x;t,y}[g]&\coloneqq E^{s,x;t,y}[e^{\beta H_I(\omega, X)-|I\cap \Z|\lambda(\beta)}g(X)].
}
For simplicity of notation, we write $Z_n^x[g]\coloneqq Z_{(0,n]}^{0,x}[g]$.

Another notation that will appear frequently is the probability measure
\begin{align}\label{eq:def_alpha}
\alpha_n(x,y)\coloneqq  \frac{Z_{n-1}^x[\1{X_n=y}]}{Z_{n-1}^x}.
\end{align}
Note that the time-horizon of the environment is $n-1$ while the position of the random walk is evaluated at time $n$. The reason for this asymmetry is that $\alpha_n$ appears in decomposition of $\log Z_n$ into a martingale part and a previsible part, hence it is natural that we want $\alpha_n$ to be $\F_{n-1}$-measurable.

\smallskip Throughout the paper, we use \(C\) to denote a large positive constant and \(c > 0\) to denote a small positive constant. Note that these constants may change from line to line. {They always depend on} the function $f$, the dimension \(d\) and the distribution of $\omega$, {in addition to any parameters (such as $p$ and $M$) mentioned in the statements}. We use $\llbracket a,b\rrbracket\coloneqq [a,b]\cap\Z$ to denote discrete intervals, and if it is clear from context, we write $n^a$ for $\lfloor n^a\rfloor$.

\subsection{Structure of the approximation}
For $f\in\mathcal C_c(\R^d)$ and $\delta\in(0,{1/6})$, let
\begin{align*}
    s_n^\delta(f)&\coloneqq n^{-d/2}\sum_{x\in \Z^d} f(x/\sqrt n)(Z_{n^{1-\delta}}^x-1),\\ 
	S_n^\delta(f)&\coloneqq n^{-d/2}\sum_{x\in \Z^d} f(x/\sqrt n)(Z_n^x-Z_{n^{1-\delta}}^x),\\
	k_n^{\delta}(f)&\coloneqq n^{-d/2}\sum_{x\in\Z^d}f(x/\sqrt n)\Big(\log Z_{n^{1-\delta}}^x-\mathbb{E}[\log Z_{n^{1-\delta}}]\Big),\\
 	K_n^{\delta}(x)&\coloneqq \log\frac{Z_n^x}{Z_{n^{1-\delta}}^x}-\E\Big[\log\frac{Z_n}{Z_{n^{1-\delta}}}\Big].
\end{align*}
In words, $s_n(f)$ and $k_n(f)$ denote the contribution from the environment close to the initial time, i.e., up to time $n^{1-\delta}$, which will be negligible for any $\delta>0$ {(see {Proposition~\ref{prop: key claims}(i)} below}). In order to prove Theorem~\ref{thm:main}, we need to show that a suitable spatial average of $K_n^\delta(x)$ and $S_n^\delta(f)$ are arbitrarily close if $\delta>0$ is small. We do this by approximating both quantities by the martingale {$M_n^{\delta}(f)\coloneqq n^{-d/2}\sum_{x\in \Z^d}f(x/\sqrt n) M_n^{\delta}(x)$, where}
\begin{align}\label{eq:M_n}
    M_n^{\delta}(x)&\coloneqq \sum_{k\in \llbracket n^{1-\delta}+1,n\rrbracket} \sum_{y\in\Z^d} \cev Z_{[k-k^{1/8},k)}^{k,y} E_{k,y}p_k(x,y),
\end{align}
and where the exponentiated, re-centered environment at $(k,y)$ is given by
\begin{align}\label{Def: Eky}
    E_{k,y}&\coloneqq e^{\beta\omega_{k,y}-\lambda(\beta)}-1.
\end{align}
\begin{prop}\label{prop: key claims}
Assume \eqref{eq:assumption} and \eqref{eq:concprop}.
\begin{enumerate}
	\item[(i)] For any $\delta\in(0,1/6)$, there exists ${\CO \eps_1=\eps_1(\delta)}>0$ such that, for any $f\in\mathcal C_c(\R^d)$ there exists $C>0$ such that, {for any $n\in\N$}, 
		\begin{align}
            \E|s_n^\delta(f)|&\leq Cn^{-\xi-{\CO \eps_1} },\label{eq:SHE_small}\\
			\E| k_n^{\delta}(f)|&\leq Cn^{-\xi-{\CO \eps_1}}.\label{eq:KPZ_small}
		\end{align}
	\item[(ii)] There exists ${\CO \eps_2}>0$ such that, for any $f\in\mathcal C_c(\R^d)$ and $\delta\in(0,1/6)$, there exists $C>0$ such that, {for any $n\in\N$},
\begin{align}
	\E| S_n^{\delta}(f)-M_n^{\delta}(f) |\leq Cn^{-\xi-{\CO \eps_2}}.\label{eq:SHE_large}
\end{align}
\item[(iii)]There exist ${\CO \eps_3}>0$, $\delta\in (0,1/6)$ and  $C>0$ such that, {for any $n\in\N$ and} $x\in\Z^d$, 
	\begin{align}
	\E| K_n^\delta(x)-M_n^\delta(x) |\leq Cn^{-\xi-{\CO \eps_3}}.\label{eq:KPZ_large}
	\end{align}
\end{enumerate}
\end{prop}

Let us see how Theorem~\ref{thm:main} follows from these estimates.
{\CO 
\begin{proof}[Proof of Theorem~\ref{thm:main} assuming Proposition~\ref{prop: key claims}]
Let $\delta$ be as in Proposition~\ref{prop: key claims}(iii) and decompose
\begin{align*}
|S_n(f)-K_n(f)|&\leq |s_n^\delta(f)|+|k_n^\delta(f)|+|S_n^\delta(f)-M_n^\delta(f)|+n^{-d/2}\sum_{x\in\Z^d}f(x/\sqrt n)|K_n^\delta(x)-M_n^\delta(x)|.
\end{align*}
Since $f$ is compactly supported, the right-hand side is bounded by $Cn^{-\xi-\eps}$, where $\eps:=\min\{\eps_1,\eps_2,\eps_3\}$ and $\eps_1,\eps_2,\eps_3$ are as in Proposition~\ref{prop: key claims}.
\end{proof}
}

\subsection{Outline}

{We start by proving Theorem~\ref{thm: lower tail concentration} in Section \ref{sec:lower}. Towards the proof of the main result, we first prove various preliminary approximation results for the random walk bridge transition probability in Section~\ref{sec:prel}, which we then use to derive approximations for the partition function, the logarithmic partition function and the polymer measure in Section~\ref{sec:approx}. Using these, the approximation of $S_n(f)$ by $M_n(f)$ is fairly straightforward and will be carried out in Section~\ref{sec:she}. Finally, the most technical part of this work is the approximation of $K_n(f)$, which is done in Section~\ref{sec:kpz}. In the appendix, we record a result about the Doob decomposition of $\log Z_n$ from \cite{CY06}.}

\section{Lower tail concentration and negative moments (Proof of Theorem~\ref{thm: lower tail concentration})}\label{sec:lower}

{To emphasize the dependence of $Z_{\omega,n}^\beta$ on $\omega$, we write it as  $Z_n(\omega)$ in this section.} We first check some properties of $\log Z_n(\omega)$, which we can interpret as a function $\R^{T_n}\to\R$, where {$T_n\coloneqq \llbracket 1, n\rrbracket\times\llbracket-n,n\rrbracket^d$}.
\begin{lem}
The function $\R^{T_n}\ni \omega\mapsto \log Z_n(\omega)$ is convex. Moreover, 
\begin{align}\label{eq:partial}
\frac{\partial}{\partial\omega_{t,x}}\log Z_n(\omega)=\beta \mu_{\omega,n}^\beta(X_t=x).
\end{align}
\end{lem}
\begin{proof}
By the H\"older inequality, for any $t\in [0,1]$ and $\omega,\omega'\in\R^{T_n}$,
        \al{
        Z_n(t\omega+(1-t)\omega')&=e^{-n\lambda(\beta)}E\left[e^{\beta\sum_{k=1}^n t\omega_{k,\S_k}+(1-t)\omega'_{k,\S_k}}\right]\\
        &\leq e^{-n\lambda(\beta)}E\left[e^{\beta\sum_{k=1}^n \omega_{k,\S_k}}\right]^{t}\cdot E\left[e^{\beta\sum_{k=1}^n \omega'_{k,\S_k}}\right]^{1-t}=Z_n(\omega)^t\cdot Z_n(\omega')^{1-t}.
        }
Next, we compute the partial derivative as follows:
\begin{align*}
    \frac{\partial}{\partial \omega_{t,x}}Z_{n}(\omega)&=\frac{\partial}{\partial \omega_{t,x}}\Big(Z_{[1,n]\setminus\{t\}}[\1{X_t=x}]e^{\beta\omega_{t,x}-\lambda(\beta)}+Z_n[\1{X_t\neq x}]\Big)=\beta Z_n[\1{X_t=x}],\\
    \frac{\partial}{\partial \omega_{t,x}}\log Z_{n}(\omega)&=\frac{1}{Z_n (\omega)} \frac{\partial Z_{n} (\omega)}{\partial \omega_{t,x}} =\frac{\beta Z_n[\1{X_t=x}]}{Z_n}=\beta \mu_{\omega,n}^\beta(X_t=x).
\end{align*}
\end{proof}
Next, we show the following  bound, which corresponds to \cite[Lemma 3.3]{CSZ20}.
\begin{lem}\label{lem:convex}
Assume \eqref{eq:assumption}. There exists $K>1$ such that, for all $n\in\N$,
\begin{align}\label{eq:event}
   \P\Big(Z_{n}(\omega)\geq \frac{1}{K},\,\sum_{k=1}^n \sum_{x\in\Z^d} \mu_{\omega,n}^\beta(X_k=x)^2\leq \frac{K}{\beta^2}\Big)\geq \frac 12.
\end{align}
\end{lem}

\begin{proof}
It is enough to find $K>1$ such that, the following hold:
\begin{align*}
    \sup_n\P\Big(Z_n(\omega)<\frac 1K\Big)&\leq \frac 14,\\
    \sup_n\P\Big(\sum_{k=1}^n \sum_{x\in\Z^d} \mu_{\omega,n}^\beta(X_k=x)^2> \frac{K}{\beta^2}\Big)&\leq \frac 14.
\end{align*}
The first claim immediately follows from the fact that $\inf_n Z_n>0$ almost surely in weak disorder. Similarly, for the second claim, recall that by \cite[Corollary~1.11(ii)]{junk2023local}, $\sup_n \sum_{k=1}^n \sum_{x\in\Z^d} \mu_{\omega,n}^\beta(X_k=x)^2<\infty$ almost surely under the assumption that $\p>1+2/d$.
\end{proof}

The following {result} is taken directly from \cite[Lemma 3.3 and Proposition 3.4]{CTT17}.
\begin{thmx}\label{thmx:conc}
Assume \eqref{eq:concprop}. There exists $C>0$ such that, for every $n\in\N$ and every convex, differentiable function $f:\R^{T_n}\to\R$ and every $a\in\R$, $t,b>0$,
\begin{align}\label{eq:adadad}
\mathbb{P}(f(\omega)\leq a-t)\P\Big(f(\omega)\geq a, \sum_{(t,x)\in T_n}\Big(\frac{\partial}{\partial\omega_{t,x}}f(\omega)\Big)^2\leq b\Big)\leq Ce^{-(t/b^2)^\gamma/C}.
\end{align}
\end{thmx}

With the help of {Theorem~\ref{thmx:conc}}, we can now prove Theorem~\ref{thm: lower tail concentration}.
\begin{proof}[Proof of Theorem~\ref{thm: lower tail concentration}]
 By Lemma~\ref{lem:convex}, $f(\omega)\coloneqq \log Z_n(\omega)$ is convex. Setting $a=-\log K$ and $b=K$, we apply Theorem~\ref{thmx:conc}. From \eqref{eq:partial} we see that the second probability in \eqref{eq:adadad} is equal to the probability in \eqref{eq:event} and that it is thus bounded from below by $\frac 12$, uniformly in $n$. We thus arrive at 
\begin{align*}
    \mathbb{P}(\log Z_n(\omega)\leq -\log K-t)\leq 2Ce^{-(t/K^2)^\gamma/C}.
\end{align*}
Next, we prove \eqref{eq:union}. For fixed $N\in\N$ and $\eps\in(0,1)$,
\begin{align*}
    \E\Big[\max_{x\in A,n\leq N}(Z_n^x)^{-k}\Big]&\leq |A|^{\eps} \Big(1+\int_1^\infty \P\Big(\max_{x\in A,n\leq N}(Z_n^x)^{-k}>t|A|^{\eps}\Big)dt\Big)\\
    &\leq |A|^{\eps} \Big(1+|A|\int_1^\infty \P\Big(\max_{n\leq N}Z_n^{-k}>t|A|^{\eps}\Big)dt\Big)\\
    & =  |A|^{\eps} \Big(1+|A|\int_1^\infty \P\Big(\max_{n\leq N}Z_n^{-k/\eps}>t^{1/\eps}|A|\Big)dt\Big)\\
    &\leq |A|^{\eps} \Big(1+\E\Big[\max_{n\leq N}Z_n^{-k/\eps}\Big]\int_1^\infty t^{-1/\eps}dt\Big),
\end{align*}
where we have used the Markov inequality in the last line. Since $\eps<1$, the last integral is finite, and by the Doob inequality and by the previous tail bound the last expectation is bounded independently of $N$. The claim thus follows by monotone convergence.
\end{proof}

\section{Preliminaries}\label{sec:prel}

In this paper, we repeatedly apply the Burkholder inequality in a specific way. To avoid repetition, we summarize the argument in the following lemma:
\begin{lem}\label{lem:bh}
{Recall $E_{k,y}$ from \eqref{Def: Eky}.}	Suppose that $A_{k,y}$ is {an $\F_{k-1}$-measurable random variable}. Let $I\subseteq \N\times \Z^d$ be such that, $I\cap \{k\}\times \Z^d$ is finite for every $k\in\N$. For every $p\in[1,2]$, there exists $C(p)>0$ such that, $N\coloneqq \sum_{(k,y)\in I}A_{k,y} E_{k,y}$ satisfies
\begin{align}\label{eq:burkholder1}
	\mathbb{E}[|N|^p]\leq C(p)\E\Big[\sum_{(k,y)\in I}|A_{k,y}|^p\Big].
\end{align}
In particular, if $A_{k,y}$ can be factorized as $A_{k,y}=A'_{k,y}A''_{k,y}$, then it holds that
\begin{align}
	\mathbb{E}[|N|^p]&\leq C(p)\E\Big[\sup_{(k,y)\in I} |A_{k,y}'|^p \sum_{(k,y)\in I}|A_{k,y}''|^p\Big],\label{eq:burkholder2}\\
	\mathbb{E}[|N|]&\leq C(p)\E\Big[\sup_{(k,y)\in I} |A_{k,y}'|^{p'}\Big]^{1/p'}\E\Big[ \sum_{(k,y)\in I} |A_{k,y}''|^p\Big]^{1/p},\label{eq:burkholder3}
\end{align}
where $p'$ is the H\"older dual of $p$ defined as $\frac{1}{p} + \frac{1}{p'} =1$. Here, $C(p)$ only depends on $p$ and on the law of $E_{k,y}$, but not on $A_{k,y}$.
\end{lem}

\begin{proof}
Let $(t_1,x_1),(t_2,x_2),\dots$ denote an enumeration of $ I$ in the lexicographical order, i.e., such that $k\mapsto t_k$ is non-decreasing, and define \(\mathcal{G}_n \coloneqq  \sigma(\omega_{t_k, x_k} \mid k \leq n)\). Let $I_n\coloneqq \{(t_k,x_k)|~k\leq n\}$ and $N_n\coloneqq \mathbb{E}[N|\G_n]=\sum_{(k,y)\in I_n}A_{k,y}E_{k,y}$, where we have used the measurability assumption in the second equality. Note that $N_n$ is a martingale with respect to $\G_n$. Let $S_n\coloneqq \sum_{k=1}^n (N_k-N_{k-1})^2= \sum_{(k,y)\in I_n} A_{k,y}^2 E_{k,y}^2$. By the Burkholder inequality \cite[Theorem 9]{B66}, there exists a universal constant $C(p)>0$ such that, 
	\begin{align}\label{eq:wq}
		\mathbb{E}[|N|^p]\leq C(p) \mathbb{E}[S_\infty^{p/2}]=C(p)\E\Big[\Big(\sum_{(k,y)\in I} A_{k,y}^2 E_{k,y}^2\Big)^{p/2}\Big]. 
	\end{align}
	Now \eqref{eq:burkholder1} follows by applying the sub-additive estimate $|\sum_i x_i|^{\theta}\leq \sum_i |x_i|^\theta$, which is valid for $\theta = p/2\in [0,1]$ and $x_i\in \R$. On the other hand, \eqref{eq:burkholder2} follows simply by noting that $E_{k,y}^2$ is non-negative, so the supremum can be pulled outside the sum in \eqref{eq:wq}. The claim again follows by the sub-additive estimate. Finally, for   \eqref{eq:burkholder3} we compute
	\begin{align*}
		\mathbb{E}[|N|]&\leq C(1) \E\Big[\Big(\sum_{(k,y)\in I}\big(A_{k,y}'A_{k,y}''\big)^2 E_{k,y}^2\Big)^{1/2}\Big]\\
			 &\leq C(1) \E\Big[\Big(\sup_{(k,y)\in I} |A_{k,y}'|\Big)\Big(\sum_{(k,y)\in I}\big(A_{k,y}''\big)^2 E_{k,y}^2\Big)^{1/2}\Big]\\ 
			 &\leq C(1)  \E\Big[\sup_{(k,y)\in I} |A_{k,y}'|^{p'}\Big]^{1/p'}\E\Big[\Big(\sum_{(k,y)\in I} \big(A_{k,y}''\big)^2 E_{k,y}^2 \Big)^{p/2}\Big]^{1/p},
	\end{align*}
	where we have used the H\"older  inequality. Now \eqref{eq:burkholder3} follows from the sub-additive estimate.
\end{proof}

We denote by $|\cdot|$  the $l_2$-norm on $\R^d$. We also frequently use the following estimate for the heat kernel of the simple random walk:
\begin{lem}\label{lem:p}
For any $p>1+2/d$, there exists $C>0$ such that, the following hold:
\begin{enumerate}
 \item[(0)] For all $n\in\N$ and $y\in\Z^d$,
    \begin{align}
            \max_{x\in\Z^d}p_n(x)&\leq Cn^{-d/2},\label{eq:well_known_max}\\
    p_n(y)&\leq Ce^{-|y|^2/(Cn)}.\label{eq:well_known_quadr}
    \end{align}
    \item[(i)]For all $N\in\N$,
\begin{align}
	\sum_{n\geq N}\sum_{y\in\Z^d}P(X_n=y)^p&\leq CN^{1-\frac d2 (p-1)}.
\end{align}
    \item[(ii)] For all $N\in\N$, $n\leq N/2$, and $y\in \Z^d$ with $|y|\leq N^{3/5}$ and $(0,0)\leftrightarrow(N,y)$,
\begin{align*}
    \sum_{k=n,\dots,N-n}\sum_{z\in \Z^d} P^{0,0;N,y}(X_k=z)^p\leq Cn^{1-\frac d2(p-1)}.
\end{align*}
    \item[(iii)] For all $N\in \N$, $n\leq 2N^{{1/7}}$, and  $y,z\in\Z^d$ with $|y|\leq N^{3/5}$ and $(0,0)\leftrightarrow(N,y)\leftrightarrow (N-n,z)$,
\begin{align}\label{eq:ratio_srw}
\Big|\frac{P(X_{N-n}=z)}{P(X_N=y)}-1\Big|\leq CN^{-1/5}.
\end{align}
\end{enumerate}
\end{lem}
\begin{proof}
Recall that $C$ is a large positive constant that may change from line to line. 

Equation \eqref{eq:well_known_max} is well-known, see for example \cite[Proposition~2.4.4]{LL10}. For \eqref{eq:well_known_quadr}, recalling that by the exponential Chebyshev inequality, we have $p_n(x)\leq e^{-nI(x/n)}$, where $I$ is the large deviation rate function of the simple random walk. The claim follows since $I(x)\geq cx^2$ for some $c>0$ and all $x\in\R^d$.

For part (i), since $\sum_{y\in \Z^d} p_n(y)=1$, we use \eqref{eq:well_known_max} to get
\begin{align*}
	\sum_{y\in\Z^d} p_n(y)^p\leq \max_y p_n(y)^{p-1}\sum_{y\in \Z^d} p_n(y)\leq C^{p-1} n^{-\frac{d(p-1)}{2}}.
\end{align*}
For $p>1+{2/d}$, the previous display is summable, and the  claim follows.

\smallskip Next, we consider part (ii). Again, using $\sum_{z\in \Z^d} p^{0,0;N,y}_k(z)=1$, we observe
\begin{align}\label{eq: hear kernel RW bridge}
\sum_{z\in \Z^d} p^{0,0;N,y}_k(z)^p&\leq \max_z p^{0,0;N,y}_k(z)^{p-1}\sum_{z\in \Z^d} p^{0,0;N,y}_k(z)=\max_z p^{0,0;N,y}_k(z)^{p-1},
\end{align}
and it is thus enough to bound $\max_z p^{0,0;N,y}_k(z)$. By the local central limit theorem for the simple random walk  \cite[Theorem~2.3.11]{LL10}, there exists $C>0$ such that, for $z\in\Z^d$ and $k\in\N$ with $(0,0)\leftrightarrow (k,z)$,
\aln{
& k^{-d/2} \exp{\left( -\frac{d|z|^2}{2k}-Ck^{-1}-C k^{-3}|z|^4\right)} \leq p_k(z)\leq k^{-d/2} \exp{\left( -\frac{d|z|^2}{2k}+Ck^{-1}+C k^{-3}|z|^4\right)}.\label{eq: LCLT}
}
Without loss of generality, we assume $k\leq N/2$ and $(0,0)\leftrightarrow (k,z)$. By \eqref{eq: LCLT}, there exists $C>0$ such that, for $|z|\leq k^{2/3}$, since we can essentially neglect the terms $Ck^{-1}$ and $C k^{-3} |z|^4$, we have
\al{
p^{0,0;N,y}_k(z) &= \frac{p_k(z) p_{N-k}(y-z)}{p_N(y)}\\
&\leq \frac{C N^{d/2}}{k^{d/2}(N-k)^{d/2}} \exp\Big(\frac{d|y|^2}{2N}-\frac{d|y-z|^2}{2(N-k)}-\frac{d|z|^2}{2k}\Big)\leq  2^d C k^{-d/2},
}
where  we have used
\al{
\frac{|y|^2}{N}-\frac{|y-z|^2}{N-k}-\frac{|z|^2}{k}=-\frac{N}{k(N-k)}\left| z- \frac{k}{N}y\right|^2\leq 0.
}
It remains to take care of $|z|> k^{2/3}$, and in this case we further need to make a distinction based on $k$. Indeed, from \eqref{eq:well_known_quadr},  if $k\ge N^{7/10}$ and $|z|>  k^{2/3}$, then since  $|z|^2/k\geq k^{1/3}\geq N^{7/30}$ and  $|y|^2/N\leq N^{1/5}= N^{6/30}$ due to $|y|\leq N^{3/5}$, we estimate
\al{
p^{0,0;N,y}_k(z)&= \frac{p_k(z) p_{N-k}(y-z)}{p_N(y)}
\leq \frac{p_k(z)}{p_N(y)}\leq C N^{d/2} \exp{\Big(-\frac{|z|^2}{C k} + \frac{C |y|^2}{N} \Big)}  \leq C e^{-k^{1/3}/C}.
}
Finally, {we consider $|z|> k^{{2/3}}$} and $k< N^{7/10}$. Since $|y|\leq N^{3/5}$ and $k\leq N/2$, We have
\al{
\frac{|y|^2}{N}-\frac{|y-z|^2}{N-k}\leq \frac{|y|^2-|y-z|^2}{N-k}
&\leq   \frac{2|y||z|}{N-k}\leq   \frac{4|z|}{N^{2/5}}.
}
{Note that, {since $k^{1/3}\leq N^{7/30}\ll N^{2/5}$},} for $N$ large enough depending on $C,$
\begin{align*}
\frac{|z|^2}{Ck} - \frac{C|z|}{N^{2/5}}=\frac{|z|^2}{2Ck} +\frac{|z|^2}{2Ck} - \frac{C|z|}{N^{2/5}}\geq \frac{k^{{4/3}}}{2Ck}+ |z|\Big(\frac{1}{2Ck^{{1/3}}} -\frac{C}{N^{2/5}}\Big)\geq \frac{k^{{1/3}}}{2C}.
\end{align*}
{Hence, since $|y-z|\leq N^{3/5}+ k \leq 2 N^{7/10}$ due to $(0,0)\leftrightarrow (k,z)$ and $|y|\leq N^{3/5}$,  for $N$ large enough,}
\al{
-\frac{|z|^2}{Ck} +C\Big(\frac{|z|}{N^{2/5}}+\frac{|y-z|^4}{N^3}+\frac{|y|^4}{N^3}\Big) &\leq -\frac{k^{{1/3}}}{2C}+C\big(N^{-{1/5}}+N^{-{3/5}}\big).} Using this together with \eqref{eq:well_known_quadr} and \eqref{eq: LCLT}, by $k\leq N/2$ and $|z|> k^{2/3}$, we get
\al{
p^{0,0;N,y}_k(z) &= \frac{p_k(z) p_{N-k}(y-z)}{p_N(y)} \\
&\leq C e^{-\frac{|z|^2}{Ck}}\times (N-k)^{-d/2} e^{-\frac{d|y-z|^2}{2(N-k)}+ \frac{C|y-z|^4}{(N-k)^{3}}} \times N^{d/2}e^{\frac{d|y|^2}{2N}+ \frac{C|y|^4}{N^3}}\\
&\leq C e^{-\frac{|z|^2}{Ck} +C\big(\frac{ |z|}{N^{2/5}}+\frac{|y-z|^4}{N^3}+\frac{|y|^4}{N^3} \big)}\\
&\leq Ce^{-k^{1/3}/C}.
}
Putting things together, there exists $C>0$ such that, if $|z|> k^{2/3}$, then for any $k \in \N$, we have
\aln{\label{Eq: srwLdpRegime}
p^{0,0;N,y}_k(z) \leq Ce^{-k^{1/3}/C} \leq C k^{-d/2}.
}
Summing up \eqref{eq: hear kernel RW bridge} over $k$, a straightforward computation ends the proof.\\

Finally, we consider part (iii). By $(N,y)\leftrightarrow (N-n,z)$, we have $|y-z|\leq |y-z|_1 \leq n\leq 2N^{{1/7}}$ and $|z|\leq |y| + |y-z|\leq C N^{3/5}$. Hence, 
\al{
\Big|-\frac{|y|^2}{N}+\frac{|z|^2}{N-n}\Big|+\frac{|y|^4}{N^3}+\frac{|z|^4}{(N-n)^3}&\leq \Big|\frac{|y|^2}{N-n}-\frac{|y|^2}{N}\Big|+\Big|\frac{|z|^2}{N-n}-\frac{|y|^2}{N-n}\Big|+CN^{-3/5}\\
&\leq \frac{2n|y|^2}{N^2}+ \frac{2|y-z|(|y|+|z|)}{N}+CN^{-3/5}\\
&\leq CN^{-{23/35}}+CN^{-{9/35}}+CN^{-3/5}\leq CN^{-1/5}.
}
Thus, by \eqref{eq: LCLT}, we  have
\begin{align*}
\Big|\frac{p_{N-n}(z)}{p_N(y)}-1\Big|
&\leq  \Big| \left(\frac{N}{N-n}\right)^{d/2}\exp(CN^{-1/5})-1\Big|\\
&\leq \big|(1+CN^{-{6/7}})^{d/2}(1+CN^{-1/5})-1\big|\leq CN^{-1/5}.
\end{align*}
\end{proof}

\section{Approximation results}\label{sec:approx}

 The purpose of this section is to prove various estimates for the partition function with a fixed starting point. The proof of the main theorem will then additionally make use of the fact that an average over starting points is taken to prove the desired convergence results. 
\subsection{Super-diffusive scale estimate}
For any event $A$ that is measurable with respect to  the sigma field of the simple random walk $X$, $x\in\Z^d$ and $n\in\N$, we {generalize \eqref{eq:def_alpha} by}
\begin{align}\label{def:alpha}
{\alpha_n(x, A)\coloneqq  \frac{Z_{n-1}^x[\1{A}]}{Z_{n-1}^x}}.
\end{align} 
For simplicity  of notation, we write $\alpha_n(A)\coloneqq \alpha_n(0,A)$. To control the contribution from parts of the environment at super-diffusive distance, the following simple lemma is useful:
\begin{lem}\label{lem:compare}
Assume \eqref{eq:assumption}.
\begin{enumerate}
    \item[(i)] For any $p\in[1,\p\wedge 2)$ there exist $C=C(p)>0$ and $k=k(p)>1$ such that, for any event $A$ that is measurable with respect to  the sigma field of the simple random walk $X$ and any $n\in\N$,
    \begin{align*}
        \mathbb{E}[Z_n[\1A]^p]\leq CP(A)^{1/k}.
    \end{align*}
    Moreover, we can take $C(1)=k(1)=1$. A similar bound holds if $Z_n$ is replaced by the point-to-point partition function $Z_{(0,n)}^{0,0;n,y}$ with $|y|\leq r n$ where $r$ is a positive constant as in Theorem~\ref{thmx:local}.
    \item[(ii)] Assume \eqref{eq:concprop} with  $\gamma>1$. There exists $C>0$ such that, for any $n,k\in\N$ and any event $A$ that is measurable with respect to  the sigma field of the simple random walk $X$,\
    \begin{align*}
        \mathbb{E}[\alpha_n(A)^k]\leq C e^{-(\log n)^{\gamma}/C}+n P(A).
    \end{align*}
\end{enumerate}
\end{lem}
\begin{proof}
\textbf{Part (i)}: For $p=1$ the claim follows  from the Fubini theorem, and the inequality is even an equality with $C=k=1$. For $p>1$, let $\delta\in(0,\frac 12)$ be small enough that $(1+\delta)p<\p\wedge 2$ and note
 \begin{align*}
 \E\big[Z_N[\1A]^p\big]&=\E\Big[Z_N[\1A]^{p(1-\delta^2)}Z_N[\1A]^{p\delta^2}\Big]\\
     &\leq \E\big[Z_N^{(1+\delta)p}\big]^{1-\delta}\E\Big[Z_N[\1A]^{p\delta}\Big]^{\delta}\\
     &\leq \E\big[Z_N^{(1+\delta)p}\big]^{1-\delta}P(A)^{p\delta^2},
 \end{align*}
 where we have used the H\"older  inequality in the second line, and the Jensen  inequality with $p\delta<1$ and part~(i) in the last line. By definition of $\p$, the first factor on the  right-hand side  is bounded in $ N$. For the point-to-point partition function, the same argument works with Theorem~\ref{thmx:local}.

 \textbf{Part (ii)}: We can estimate {
 \begin{align*}
     \mathbb{E}[\alpha_n(A)^k]\leq \mathbb{E}[\alpha_n(A)] \leq \mathbb{P}(Z_{n-1}\leq n^{-1})+n\mathbb{E}[Z_{n-1}[\1A]],
 \end{align*}}
 and use part~(i) with $p=1$ and Theorem~\ref{thm: lower tail concentration}.
\end{proof}

\subsection{Approximations for the partition function}\label{sec:approx_part}

\begin{prop}\label{prop:diff}
Assume \eqref{eq:assumption}.
\begin{enumerate}
	\item[(i)] For every $p\in(1+2/d,\p\wedge 2)$ there exists $C>0$ such that, for every $0\leq n\leq N$ and for every function $g$  measurable with respect to  the filtration of the random walk $X$, 
\begin{align}
	\mathbb{E}[|Z_n-Z_N|^p]&\leq Cn^{1-\frac d2(p-1)},\label{eq:a}\\ 
	{\E\Big[\big|Z_n[g]-Z_N[g]\big|^p\Big]}&\leq C\|g\|_\infty^p n^{1-\frac d2(p-1)}. \label{eq:b}
\end{align}
\item[(ii)] For every $p\in(1+2/d,\p\wedge 2)$ there exists $C>0$ such that, for every $n\in\N$, $m\in \Iintv{1,n/2}$, and $|y|\leq n^{3/5}$ with $(0,0)\leftrightarrow (n,y)$,
	\begin{align}
  \mathbb{E}[|Z_{(0,m]\cup[n-m,n)}^{0,0;n,y}-Z_{(0,n)}^{0,0;n,y}|^p]\leq Cm^{1-\frac d2(p-1)}.\label{eq:error1}
	\end{align}
 Moreover, for any $m\in \llbracket 1,n^{1/7}\rrbracket$, almost surely,
 \begin{align}
     \Big|\frac{Z_{(0,m]\cup[n-m,n)}^{0,0;n,y}}{Z_{m}\cev Z^{n,y}_{[n-m,n)}}-1\Big|\leq Cn^{-1/5}.\label{eq:error2}
 \end{align}
\end{enumerate}\end{prop}
\begin{remark}
In \cite{comets2017rate,CN21}, it is proved that in the $L^2$-regime, i.e., {for $\beta<\beta_2$}, the speed of convergence of $Z_n$ to $Z_\infty$ is $n^{-\frac{d-2}4}$ (more precisely, it is shown that $n^{\frac{d-2}4}(Z_n-Z_\infty)$ converges in law to an explicit limiting random variable). The bound \eqref{eq:a} suggests that outside the $L^2$-regime, i.e., {for $\beta\in (\beta_2,\beta_c)$}, the speed of convergence should be $n^{-\xi+o(1)}$, where $\xi$ is as in \eqref{eq:xi}. Indeed, given $\eps>0$, if we choose $p\in(1+2/d,\p\wedge 2)$ {sufficiently close to $\p$,} then \eqref{eq:a} yields
\begin{align*}
    \mathbb{P}(|Z_n-Z_\infty|>n^{-\xi+\eps})\leq n^{(\xi-\eps)p} \E |Z_n-Z_\infty|^p \leq Cn^{-\eps/2}.
\end{align*}
A corresponding lower bound, as well as the speed of convergence in the critical case $\p=1+2/d$, are left for future research.
\end{remark}

\begin{proof}
	\textbf{Part (i)}: 
 Since \eqref{eq:a} is a special case of \eqref{eq:b} with $g \equiv 1$, we will focus on \eqref{eq:b}. For $r\in(0,1)$, let $A\coloneqq \{|X_k|>kr\text{ for some }n\leq k\leq N\}$. We have
 \begin{align}\label{eq:write}
     \mathbb{E}[|Z_N[g]-Z_n[g]|^p]\leq 3\|g\|_\infty^p\E\big[Z_N[\1A]^p\big]+3\|g\|_\infty^p\E\big[Z_n[\1A]^p\big]+3\E\Big[\big|Z_N[g \1{A^c}]-Z_n[g \1{A^c}]\big|^p\Big].
 \end{align}
The first two terms decay exponentially fast in $n$ by Lemma~\ref{lem:compare}. To treat the final term, we apply Lemma~\ref{lem:bh} with $A_{k,y}\coloneqq Z_{(0,k)}^{0,0;k,y}[\1{A^c} g]p_k(y)$, which gives
    \begin{align*}
	    \mathbb{E}[|Z_N[\1{A^c}g]-Z_n[\1{A^c}g]|^p]\leq C(p)  \E\Big[\sum_{t=n+1}^{N}\sum_{|y|\leq r t} |Z_{(0,t)}^{0,0;t,y}[\1{A^c} g]|^p p_t(y)^p\Big].
    \end{align*}
    Since $\E [|Z_{(0,t)}^{0,0;t,y}[\1{A^c} g]|^p]\leq \|g\|_\infty^p\E [(Z_{(0,t)}^{0,0;t,y})^p]$ is, by {Theorem~\ref{thmx:local}}, bounded over $t\in\N$ and $|y|\leq r t$ with a suitable $r>0$, by Lemma~\ref{lem:p}, we get
    \begin{align*}
	    \mathbb{E}[|Z_N-Z_n|^p]\leq C\|g\|_\infty^pe^{-n/C}+C\|g\|_\infty^p\sum_{t=n+1}^{N} \sum_{y\in\Z^d}p_t(y)^p\leq  C\|g\|_\infty^p n^{1-\frac d2(p-1)}.
    \end{align*}
   
\smallskip \textbf{Part (ii)}: The first bound \eqref{eq:error1} is similar to part (i), so 
we outline how the computation can be modified. First, we set $\widetilde A\coloneqq \{|X_k|>rk\text{ or }|X_k-y|>r(k-n)\text{ for some }k\in\Iintv{m,n-m}\}$ with $r$ as in Theorem~\ref{thmx:local}.  By \eqref{Eq: srwLdpRegime}, for $n$ large enough, we estimate
\al{
P(\widetilde A)&\leq 2\sum_{k=\Iintv{m,n/2}} \sum_{z\in \Z^d:~|z| > r k/2} p^{0,0;N,y}_k(z)\\
&\leq \sum_{k=\Iintv{m,n/2}} \sum_{z\in[-n,n]^d\cap  \Z^d}  C e^{-k^{1/3}/C}
\leq  C e^{-m^{1/3}/C}.
}
Hence, as in part (i), we obtain
\begin{align*}
    \mathbb{E}[|Z_{(0,m]\cup[n-m,n)}^{0,0;n,y}-Z_{(0,n)}^{0,0;n,y}|^p]\leq Ce^{-m^{{1/3}}/C}+C\mathbb{E}[|Z_{(0,m]\cup[n-m,n)}^{0,0;n,y}[\1{\widetilde A^c}]-Z_{(0,n)}^{0,0;n,y}[\1{\widetilde A^c}]|^p].
\end{align*}
The rest of the calculation is as before: noting
\al{
&\mathbb{E}[Z_{(0,{m}]\cup[n-m,n)}^{0,0;n,y}[\1{\widetilde A^c}]-Z_{(0,n)}^{0,0;n,y}[\1{\widetilde A^c}]\\
&= \sum_{t=m+1}^{n-m-1}\sum_{z:|z|\leq r t,|z-y|\leq r(n-t)} Z_{(0,t)\cup [n-m,n)}^{0,0;n,y}[\1{\widetilde A^c\cap \{X_t=z\}}]E_{t,z},
}
by Lemma~\ref{lem:bh} and Lemma~\ref{lem:p}(ii), we estimate 
\begin{align*}
&\mathbb{E}[|Z_{(0,m]\cup [n-m,n)}^{0,0;n,y}[\1{\widetilde A^c}]-Z_{(0,n)}^{0,0;n,y}[\1{\widetilde A^c}]|^p]\\
& \leq C \E\Big[\sum_{t=m+1}^{n-m-1}\sum_{z:|z|\leq r t,|z-y|\leq r(n-t)}\Big(Z_{(0,t)}^{0,0;t,z}Z^{t,z;n,y}_{[n-m,n)}\Big)^p {P^{0,0;n,y}(X_t=z)}^{p}\Big]\\
&\leq C \sum_{t=m+1}^{n-m-1}\sum_{z\in\Z^d} P^{0,0;n,y}(X_t=z)^p\leq C m^{1-\frac{d}{2}(p-1)}.
\end{align*}
For \eqref{eq:error2}, we use \eqref{eq:ratio_srw} to verify that
    \begin{align*}
   & \Big|Z_{(0,m]\cup[n-m,n)}^{0,0;n,y}-Z_{m}\cev Z_{[n-m,n)}^{n,y}\Big|\\
   &\leq \sum_{|z_1|\leq m,|y-z_2|\leq m} Z_{(0,m]}^{0,0;m,z_1}Z_{[n-m,n)}^{n-m,z_2;n,y}p_{m}(z_1)p_{m}(z_2-y)\Big| \frac{p_{n-2m}(z_1-z_2)}{p_n(y)}-1\Big|\\
   &\leq Cn^{-1/5}Z_{m}\cev Z_{[n-m,n)}^{n,y}.
    \end{align*}
\end{proof}
\subsection{Approximations for the polymer measure}\label{sec:approx_polymeas}

The purpose of this section is to approximate the polymer measure $\alpha_k(x,y)$ (recall \eqref{eq:def_alpha}) by
\begin{align}\label{eq:def_alphatilde}
\widetilde \alpha_k(x,y)\coloneqq \Big(p_k(y-x)\cev Z_{[k-k^{{1/8}},k)}^{k,y}\Big)\wedge 1.
\end{align}
For simplicity of notation, we write $\alpha_k(y)\coloneqq \alpha_k(0,y)$ and $\tilde{\alpha}_k(y)\coloneqq\tilde{\alpha}_k(0,y)$.
\begin{lem}\label{lem:ratio}
Assume \eqref{eq:assumption} and \eqref{eq:concprop}. 
\begin{enumerate}
    \item[(i)]There exists $\eps>0$ such that, for all  $M\in\N$ there exists $C>0$ such that, for all $n\in  \N$, 
\begin{align}
	\sup_{m\geq n}\mathbb{E}\Big[\Big|\frac{Z_{n}}{Z_m}-1\Big|^M\Big]\leq C n^{-\eps}.
\end{align}
 \item[(ii)]There exist $\eps>0$ and $C>0$ such that, for all $n\in\N$, $|y|\leq n^{3/5}$, and $p\geq \p\wedge 2$,
	\begin{align*}
		\E\Big[\big|\alpha_n(y)-\widetilde \alpha_n(y)\big|^p\Big]\leq Cn^{-\frac{(\p\wedge 2) d}2-\eps}.
	\end{align*}
 \item[(iii)] There exists $\eps>0$ such that, for any $p\in (1+2/d,\p\wedge 2)$, there exists  $C>0$ such that, for all $n\in\N$ and  $|y|\leq n^{3/5}$,
	\begin{align*}
		\E\Big[\big|\alpha_n(y)-p_n(y)\cev Z_{[n-n^{{1/8}},n)}^{n,y}\big|^p\Big]\leq Cn^{-\frac{p d}2-\eps}.
	\end{align*}
\end{enumerate}

\end{lem}

\begin{proof}
For \textbf{part (i)}, we note that if $x\geq 0$ and $a\in(0,1/2)$ are such that, $|x^{-1}-1|\leq a$, then since $1/(1+a)\leq x\leq 1/(1-a)$, 
$$-2a\leq \frac{-a}{1-a}\leq x-1\leq  \frac{a}{1+a}\leq 2a.$$ Thus, together with the Cauchy-Schwarz inequality, for $\delta>0$,
\begin{align*}
	\mathbb{E}\Big[\Big|\frac{Z_{n}}{Z_m}-1\Big|^M\Big]&\leq (2n^{ -\delta})^M +\mathbb{E}\Big[\Big|\frac{Z_{n}}{Z_m}-1\Big|^{2M}\Big]^{1/2}\mathbb{P}\Big(\Big|\frac{Z_m}{Z_{n}}-1\Big|\geq n^{-\delta}\Big)^{1/2}.
\end{align*}
Note that $(2n^{ -\delta})^M\leq C n^{-\delta}$. By the Jensen  inequality, we have
\al{
\Big(\frac{Z_{n}}{Z_m}-1\Big)^{2M}&\leq \Big(\frac{Z_{n}}{Z_m}\Big)^{2M} +1 = \Big( \sum_{y\in\Z^d} \frac{Z_n[\mathbf{1}_{\{X_n=y\}}]}{Z_n} Z_{(m-n,m]}^{m-n,y} \Big)^{-2M} +1\\
&\leq \sum_{y\in\Z^d} \frac{Z_n[\mathbf{1}_{\{X_n=y\}}]}{Z_n} ( Z_{(m-n,m]}^{m-n,y})^{-2M}+1.
}
Together with the bound on negative moments \eqref{eq:union}, we obtain
\begin{align*}
	\mathbb{E}\Big[\Big|\frac{Z_{n}}{Z_m}-1\Big|^{2M}\Big]\leq \mathbb{E}\Big[\big(Z_{m-n}\big)^{-2M}+1\big]\leq\sup_{k\geq 0}\mathbb{E}[Z_{k}^{-2M}]+1<\infty.
\end{align*}
  By Theorem~\ref{thm: lower tail concentration} and Proposition~\ref{prop:diff}(i), fixing an arbitrary parameter $p$ in $(1+d/2,\p\wedge 2)$, if $\delta>0$ is sufficiently small, then we have  
\begin{align*}
	\mathbb{P}\Big(\Big|\frac{Z_m}{Z_{n}}-1\Big|\geq n^{-\delta}\Big)&\leq \mathbb{P}(Z_{n}<n^{-\delta})+\mathbb{P}\Big(\big|Z_m-Z_{n}\big|\geq n^{-2\delta}\Big)\\
 &\leq C e^{-(\delta \log n)^\gamma/C}+n^{2\delta p} \mathbb{E}[\big|Z_m-Z_{n}\big|^p] \leq C n^{-\eta},
\end{align*}
with some $\eta = \eta(p,\delta)>0$, which yields the claim.

\smallskip For \textbf{part (ii)} and  \textbf{part (iii)}, we fix $p_0\in(1+2/d,\p\wedge 2)$. Let $p\in(p_0,\p\wedge 2)$. For any $q\geq \p\wedge 2$, we have
\begin{align*}
\E\big[\big|\alpha_n(y)-\widetilde \alpha_n(y)\big|^q \big]
\leq \mathbb{E}[\big|\alpha_n(y)-\widetilde \alpha_n(y)\big|^p]\leq \mathbb{E}[\big|\alpha_n(y)-\cev Z_{[n-n^{{1/8}},n)}^{n,y}p_n(y)\big|^p].
\end{align*}
In the first inequality, we have used that both $\alpha$ and $\widetilde\alpha$ are bounded by one, and in the second inequality we have used the general fact that $|a-(b\wedge 1)|\leq |a-b|$ for $a\in[0,1]$ and $b\geq 0$. To bound the last expectation, by using $(a+b)^p\leq 2 a^p + 2b^p$ for $a,b\geq 0$ and $p\in [1,2]$, we compute, for $\delta>0$ small enough that $(1+\delta)p<\p\wedge 2$,
\begin{align*}
&\mathbb{E}[\big|\alpha_n(y)-\cev Z_{[n-n^{{1/8}},n
)}^{n,y}p_n(y)\big|^p]\\
&\leq 2 p_n(y)^p\Big(\E\Big[\Big|\frac{Z_{(0,n)}^{0,0;n,y}-Z_{n^{{1/8}}}\cev Z_{[n-n^{{1/8}},n)}^{n,y}}{Z_{n}}\Big|^p\Big]+\E\Big[\Big|\frac{Z_{n^{{1/8}}}}{Z_n}-1\Big|^p(\cev Z_{[n-n^{{1/8}},n)}^{n,y})^p\Big]\Big)\\
&\leq 2 p_n(y)^p\Big(\E\Big[Z_{n}^{-\frac{(1+\delta)p}{\delta}}\Big]^{\frac{\delta}{1+\delta}}\E\Big[\big|Z_{(0,n)}^{0,0;n,y}-Z_{n^{{1/8}}}\cev Z_{[n-{n^{{1/8}}},n)}^{n,y}\big|^{(1+\delta)p}\Big]^{\frac{1}{1+\delta}}\\
&\qquad\qquad +\E\Big[\Big|\frac{Z_{n^{{1/8}}}}{Z_n}-1\Big|^{\frac{(1+\delta)p}{\delta}}\Big]^{\frac{\delta}{1+\delta}} \E\Big[(\cev Z_{[n-n^{{1/8}},n)}^{n,y})^{(1+\delta)p}\Big]^{\frac{1}{1+\delta}}\Big)\\
&\leq C(p,\delta) n^{-dp/2} n^{-\eta},
\end{align*}
where in the last line we have used Lemma~\ref{lem:p}, \eqref{eq:union}, part (i) in this lemma and Proposition~\ref{prop:diff}(ii), and the exponent $\eta$ is positive and  depends only on $p_0$ but not on $p$ and $\delta$. Thus, part (ii) follows by choosing $\eps \coloneqq  \eta/2$ and $p<\p\wedge 2$ large enough  that $\frac{pd}2+\eta\geq \frac{(\p\wedge 2) d}2+\frac \eta {2}$.
\end{proof}

\begin{prop}\label{prop:alpha2}
Assume \eqref{eq:assumption} and \eqref{eq:concprop}. For any $\eps>0$ there exist $n_0\in\N$ and $C>0$ such that, for all $n\geq n_0$,
\begin{align}
    \sup_{y\in \Z^d}\mathbb{E}[\alpha_n(y)^2]&\leq Cn^{-\frac{d}{2}(\p\wedge 2)+\eps},\label{eq:max_alpha}\\
    \sum_{y\in\Z^d}\mathbb{E}[\alpha_n(y)^2]&\leq Cn^{-\frac d2((\p\wedge 2)-1)+\eps}\label{eq:sum_alpha}.
\end{align}
{In particular, for any $\eta\in(0,1)$ there exist $C,\varepsilon>0$ such that, for all $n\in\N$, }
\begin{align}\label{eq:partial_sum}
{\sum_{k\geq n^{\frac 1{\p\wedge2}+\eta}}\sum_{y\in\Z^d}\mathbb{E}[\alpha_k(y)^2]\leq
C n^{-\xi-\varepsilon}.}
\end{align}
\end{prop}
\begin{proof}
 We first consider $y\in\Z^d$ with $|y|>\sqrt{n}\log{n}$. Since $p_n(y)\leq C e^{-(\log{n})^2/C}$ due to  \eqref{eq:well_known_quadr},  by Theorem~\ref{thm: lower tail concentration} and $\mathbb{E}[Z_{(0,n)}^{0,0;n,y}]=1$, 
 for any $u\leq 1$,
\al{
\mathbb{P}(\alpha_n(y)\geq u)&\leq \mathbb{P}\Big(Z_{(0,n)}^{0,0;n,y}p_{n}(y)\geq u e^{-(\log{n})^{2/\gamma}}\Big)+\mathbb{P}\Big(  Z_{n-1}\leq e^{-(\log{n})^{2/\gamma}}\Big)\\
&\leq C u^{-1}  e^{-(\log{n})^2/C} e^{(\log{n})^{2/\gamma}}+C e^{-(\log{n})^{2}/C}\\
&\leq C u^{-1}e^{-(\log{n})^{2}/C}.
}
We thus have, for $n$ large enough,
\begin{equation}\label{eq:morethan}
\begin{split}
\mathbb{E}[\alpha_n(y)^2]&\leq 2\int_{0}^1 u \mathbb{P}(\alpha_n(y)\geq u)\dd u \leq C e^{-(\log{n})^{2}/C}.
\end{split}
\end{equation}
which is more than sufficient for \eqref{eq:max_alpha}. Next, we consider $|y|\leq \sqrt{n}\log{n}$. We fix  $p\in(1,\p\wedge 2)$ arbitrary. By {Theorem~\ref{thmx:local}},    for any $u>0$,
\aln{\label{p-the moment of p2p}
\mathbb{P}(Z_{(0,n)}^{0,0;n,y}\geq u)\leq u^{-p}\mathbb{E}[(Z_{(0,n)}^{0,0;n,y})^p]\leq C u^{-p}.
}
Due to Theorem~\ref{thm: lower tail concentration}, since $p_n(y)\leq Cn^{-d/2}$ by Lemma~\ref{lem:p}(0), we have for any $u\leq 1$,
\al{
\mathbb{P}(\alpha_n(y)\geq u)&\leq \mathbb{P}(Z_{(0,n)}^{0,0;n,y}\geq C^{-1} u n^{d/2}e^{-(\log{n})^{{2/(1+\gamma)}}})+\mathbb{P}( Z_{n-1}\leq e^{-(\log{n})^{2/(1+\gamma)}})\\
&\leq C u^{-p} n^{-pd/2}e^{p(\log{n})^{2/(1+\gamma)}}+Ce^{-C^{-1}(\log{n})^{2\gamma/(1+\gamma)}}.
}
Therefore, since $\gamma>1$, for $n\in\N$ large enough, we have uniformly in $y$,
\al{
\mathbb{E}[\alpha_{n}(y)^2]&=\int_0^1 2u\mathbb{P}(\alpha_n(y)\geq u){\rm d} u\\
&\leq 2C n^{-pd/2}e^{p(\log{n})^{2/(1+\gamma)}}\int_0^1 u^{1-p}{\rm d} u+2C e^{-C^{-1}(\log{n})^{2\gamma/(1+\gamma)}}\\
&\leq \frac{4C}{2-p} n^{-pd/2}e^{p(\log{n})^{2/(1+\gamma)}}.
}
By choosing $p$ sufficiently close to { $\p\wedge 2$}, we see that the final expression is bounded by { $Cn^{-\frac {(\p\wedge 2) d}2+\eps}$} for all $n$ large enough, which proves \eqref{eq:max_alpha}. Next, we prove \eqref{eq:sum_alpha}. By \eqref{eq:max_alpha}, we have
\al{
\E\Big[\sum_{|y|\leq n^{1/2}\log n} \alpha_{n}(y)^2\Big]\leq C(\log n)^d\,n^{\frac{d}{2}-\frac{(\p\wedge 2) d}{2}+\eps}.
}
On the other hand, using \eqref{eq:morethan}, 
\begin{align*}
\E\Big[\sum_{|y|> n^{1/2}\log n} \alpha_{n}(y)^2\Big]\leq C n^d \max_{|y|>n^{1/2}\log(n)}\mathbb{E}[\alpha_n(y)^2]\leq C n^d e^{-(\log{n})^{2}/C},
\end{align*}
which is again more than sufficient for \eqref{eq:sum_alpha}.
\end{proof}

\subsection{Approximations for the $\log$-partition function}
{In this subsection we derive an estimate similar to Lemma~\ref{lem:compare}(ii) for the logarithm of the partition function. Namely, we show that the logarithm partition function essentially depends only on the environment in a diffusive space-time window and that the effect from the outside of this window is very small. This is very technical but crucial for the analysis of the KPZ equation.}

\begin{prop}\label{prop: log correlation}
Assume \eqref{eq:assumption} and \eqref{eq:concprop}. 
\begin{enumerate}
    \item[(i)] For any $m\in\N$,
    \aln{\label{Eq: finite log moment}
    \sup_{n\in\N}\mathbb{E}[|\log{Z_n}|^m]<\infty.
    }
\item[(ii)]
There exists $C>0$ such that, for any $x,y\in\Z^d$, $\eta>0$, and any $n\in\N$ large enough,
\begin{alignat*}{2}
|\mathbb{E}[ A_n(x)A_n(y)]|
    &\leq Cn^{\frac{d}{2}-\frac{(\p\wedge 2) d}{2}+\eta} &\qquad\text{ if }|x-y|\leq \sqrt{n}\,\log{n},\\
        |\mathbb{E}[ A_n(x)A_n(y)]|&\leq Ce^{-(\log{n})^{\gamma}/C} &\qquad\text{ if }|x-y|> \sqrt{n}\,\log{n},
    \end{alignat*}
    where $A_n(x)$ is either $\log \frac{Z^x_{n}}{Z^x_{n-1}}-\mathbb{E}[\log \frac{Z^x_{n}}{Z^x_{n-1}}|\F_{n-1}]$ or $\mathbb{E}[\log \frac{Z^x_{n}}{Z^x_{n-1}}|\F_{n-1}]- \mathbb{E}[\log \frac{Z^x_{n}}{Z^x_{n-1}}]$.
\end{enumerate}
\end{prop}
\begin{proof}
For part (i), we first note that by the Markov inequality,  for any $t\geq 0$,
    $$\mathbb{P}(\log{Z_n}\geq t)=\mathbb{P}(Z_n\geq e^t)\leq e^{-t}\mathbb{E}[Z_n]=e^{-t}.$$
    Moreover, by Theorem~\ref{thm: lower tail concentration}, there exists $C>0$ independent of $n$ such that, for any $t\geq 0$
    $$\mathbb{P}(\log{Z_n}\leq -t)\leq C e^{-t^\gamma /C}.$$
    Therefore, $\sup_{n\in\N}\mathbb{P}(|\log Z_n|>t)$ decays super-polynomially as $t\to\infty$. Thus, the claim follows.
    
For the first line of part (ii), by the Cauchy-{Schwarz} inequality, it suffices to prove that the desired bound holds for  $\mathbb{E}[(\log{\frac{Z^x_{n}}{Z^x_{n-1}}})^2]$. Indeed, by \eqref{upper ineq for U},
\al{
\E\Big[\Big(\log{\frac{Z^x_{n}}{Z^x_{n-1}}}\Big)^2\Big]&=
\E\Big[\Big(\log{\Big(1+\sum_{y\in \Z^d} \alpha_{n}(x,y)E_{n,y}\Big)}\Big)^2\Big]\leq C \E\Big[\sum_{y\in \Z^d} \alpha_{n}(x,y)^2\Big].
}
By Proposition~\ref{prop:alpha2}, we have the first line bound.

\smallskip Next, we consider the second line, i.e., we assume $|x-y|> \sqrt{n}\,\log{n}$. The main idea is that the contribution to the partition functions from paths not confined within a cone of diameter $n^{1/2}\log{n}$ is negligible, so the two factors are approximately independent. More precisely, we introduce 
\begin{align*}
\hat{Z}^x_{n}&\coloneqq Z_n^x[\1{\,|X_k-x|\leq \sqrt{n}\,\log{n}/2\text{ for all }k\in{\llbracket 1,n\rrbracket}}],\\
\hat{Z}^x_{n-1}&\coloneqq Z_{n-1}^x[\1{\,|X_k-x|\leq \sqrt{n}\,\log{n}/2\text{ for all }k\in{\llbracket 1,n\rrbracket}}].\end{align*}
If $\hat A_n(x)$ is equal to either $\log \frac{\hat{Z}^x_{n}}{\hat{Z}^x_{n-1}}-\mathbb{E}[\log \frac{\hat{Z}^x_{n}}{\hat{Z}^x_{n-1}}|\F_{n-1}]$ or $\mathbb{E}[\log \frac{\hat{Z}^x_{n}}{\hat{Z}^x_{n-1}}|\F_{n-1}]- \mathbb{E}[\log \frac{\hat{Z}^x_{n}}{\hat{Z}^x_{n-1}}]$,  then $\hat A_n(x)$ and $\hat A_n(y)$ are centered and independent. Hence, by the Cauchy-Schwarz inequality, we have
\begin{align*}
|\mathbb{E}[A_n(x)A_n(y)]|&=\big|\mathbb{E}[A_n(x)A_n(y)-\hat A_n(x)\hat A_n(y)]\big|\\
&\leq \E\big[ |A_n(x)||A_n(y)-\hat A_n(y)|\big]+\E\big[ |\hat A_n(y)||A_n(x)-\hat A_n(x)|\big]\\
&\leq \big(\mathbb{E}[A_n(x)^2]^{1/2}+\mathbb{E}[\hat A_n(x)^2]^{1/2}\big)\mathbb{E}[(A_n(x)-\hat{A}_n(x))^2]^{1/2},
\end{align*}
where in the last line we have also used that the law{s} of $A_n(x)$ and $\hat A_n(x)$ {do} not depend on $x$. It is thus enough to obtain a bound for the last terms. Moreover, using the conditional Jensen inequality and Proposition~\ref{prop: log correlation}(i), it is enough to show that 
\begin{align}\label{eq:tildeA}
\mathbb{E}\Big[\Big(\log \frac{Z^x_{n}}{Z^x_{n-1}}-\log \frac{\hat{Z}^x_{n}}{\hat{Z}^x_{n-1}}\Big)^2\Big]&\leq C  e^{-(\log{n})^{\gamma}/C},\\
\E\Big[\Big(\log \frac{\hat{Z}^x_{n}}{\hat{Z}^x_{n-1}}\Big)^2\Big]&\leq C.\label{eq:tildeAA}
\end{align}
To prove \eqref{eq:tildeA}, recalling that $\log(1+x)\leq x^{1/2}$ for all $x\geq 0$, we have
\begin{align*}
\log{Z^x_{n}}-\log{\hat{Z}^x_{n}}&=\log{\Big(1+\frac{Z^x_{n}-\hat{Z}^x_{n}}{\hat{Z}^x_{n}}\Big)}\leq \Big(\frac{Z^x_{n}-\hat{Z}^x_{n}}{\hat{Z}^x_{n}}\Big)^{1/2}.
\end{align*}
The same argument applies to $Z^x_{n-1}$ and $\hat{Z}^x_{n-1}$. Writing $B\coloneqq \{\hat{Z}^x_{n}\leq 1/n\text{ or }\hat{Z}^x_{n-1}\leq 1/n\}$, we arrive at
\begin{equation}\label{eq:finally}\begin{split}
    &\mathbb{E}\Big[\Big(\log \frac{Z^x_{n}}{Z^x_{n-1}}-\log \frac{\hat{Z}^x_{n}}{\hat{Z}^x_{n-1}}\Big)^2\Big]\\
    &\leq 2 n\E\big[(Z^x_{n}-\hat{Z}^x_{n})\big]  +2 n\E\big[(Z^x_{n-1}-\hat{Z}^x_{n-1})\big]\\
    &\qquad +\mathbb{E}\Big[\Big(\log \frac{Z^x_{n}}{Z^x_{n-1}}-\log \frac{\hat{Z}^x_{n}}{\hat{Z}^x_{n-1}}\Big)^4\Big]^{1/2}\mathbb{P}(B)^{1/2}.
\end{split}\end{equation}
To control the first term on the  right-hand side, note that by Lemma~\ref{lem:compare},
\begin{align*}\mathbb{E}[Z^x_{n}-\hat{Z}^x_{n}]&= \mathbb{E}[{Z}_n^x[\mathbf{1}_{\{\exists k\in{\llbracket 1,n\rrbracket},\,|S_k-x|> \sqrt{n}\,\log{n}/2\}}]]\leq C e^{- (\log{n})^2/C}.
\end{align*}
The second term is bounded similarly. Moreover, from this bound and Theorem~\ref{thm: lower tail concentration}, we get
\begin{equation}\label{eq:probbound}\begin{split}
     \mathbb{P}(B)&\leq \mathbb{P}(Z_n<2/n)+\mathbb{P}(Z_{n-1}<2/n)+\mathbb{P}(Z_n-\hat Z_n\geq 1/n)+\mathbb{P}(Z_{n-1}-\hat Z_{n-1}\geq 1/n)\\
     &\leq C e^{-(\log{n})^\gamma /C}.
\end{split}\end{equation}
Next, let 
us define the probability measure        that is measurable in $\kF_{n-1}$:
\begin{align*}
\hat{\alpha}_{n}(y)\coloneqq  Z_{n-1}^x[\mathbf{1} \{|X_k-x|\leq \sqrt{n}\,\log{n}/2\text{ for all }k\in{\llbracket 1,n\rrbracket},\,X_n=y\}]/\hat{Z}^x_{n-1}.
\end{align*}    
 Since $\log x \leq \log(1+x)\leq x^{1/2}$ for $x>0$, by the Jensen  inequality,
    we have
\begin{align*}
 \sum_{y\in \Z^d} \hat{\alpha}_{n}(y) (\beta\omega_{n,y}-\lambda(\beta)) &\leq \log \frac{\hat{Z}^x_{n}}{\hat{Z}^x_{n-1}}=\log\left(\sum_{y\in \Z^d} \hat{\alpha}_{n}(y) e^{\beta\omega_{n,y}-\lambda(\beta)}\right)\\
 &\leq \left(\sum_{y\in \Z^d} \hat{\alpha}_{n}(y) e^{\beta\omega_{n,y}-\lambda(\beta)}\right)^{1/2}. 
    \end{align*}
Thus, since $\mathbb{E}[e^{\beta \omega_{n,y}-\lambda(\beta)}]=1$, by the Jensen inequality, we have 
\begin{align*}
   \E\Big[\Big(\log\frac{\hat{Z}^x_{n}}{\hat{Z}^x_{n-1}}\Big)^4\Big]\leq \mathbb{E}[(\beta \omega_{1,0}-\lambda(\beta))^4]+\mathbb{E}[e^{2\beta\omega_{1,0}-2\lambda(\beta)}]<\infty.
\end{align*}
This implies \eqref{eq:tildeAA}. Moreover, $\mathbb{E}[(\log \frac{Z^x_{n}}{Z^x_{n-1}})^4]$ is independent of $x$ and bounded in $n$  due to \eqref{Eq: finite log moment}. By \eqref{eq:probbound}, the final term in \eqref{eq:finally} is bounded by $Ce^{-(\log n)^\gamma/C}$. This finishes the proof of \eqref{eq:tildeA}. 
\end{proof}

\subsection{Homogenization at small times}
The following result shows that two quantities of interest (related to $S_n(f)$ and $K_n(f)$) homogenize if their time-horizon is of order $n^{1-\delta}$ with some $\delta>0$ and a spatial average is taken over the diffusive scale $n^{1/2}$. In the case of $S_n(f)$, the quantity of interest is simply the partition function itself, whereas in the case of $K_n(f)$, we bound the main contribution of the martingale part of $\log Z_n^x$, namely
\begin{align}\label{eq:def_Khat}
	\widehat K_n(x)\coloneqq \sum_{k=1}^n\sum_{y\in\Z^d} \alpha_k(x,y)E_{k,y}.
\end{align}
We obtain the following bound.
\begin{lem}\label{lem: martingale estimate until Ln}
 Assume \eqref{eq:assumption} and \eqref{eq:concprop}.   For any $\delta>0$, there exists $\eta>0$ such that, for all $n\in\N$ large enough and all $0\leq m< M\leq n^{1-\delta}$
    \begin{align}
        \E\Big[\Big|n^{-d/2}\sum_{x\in\Z^d}f(x/\sqrt n)(Z_M^x-Z_m^x)\Big|\Big]\leq n^{-\xi-\eta},\label{eq:aa}\\
   \E\Big[\Big|n^{-d/2}\sum_{x\in\Z^d} f(x/\sqrt n)(\widehat K_M(x)-\widehat K_m(x))\Big|\Big]\leq n^{-\xi-\eta}.\label{eq:bb}
    \end{align}
\end{lem}
{Note that by taking $m=0$ and $M=n^{1-\delta}$ in \eqref{eq:aa}, we have proved \eqref{eq:SHE_small}.}
\begin{proof}
In both cases, it is easy to take care of the contributions from $|y|\geq n^{1/2}\log n$. Indeed, in the case of \eqref{eq:aa}, we note that $Z_M^x-Z_{ m}^x=\sum_{k=m+1}^M\sum_{y\in\Z^d}Z_{k-1}^x[\1{X_k=y}]E_{k,y}$ and that
\begin{align*}
&	\E\Big[\Big|n^{-d/2}\sum_{x\in\Z^d} f(x/\sqrt n)\sum_{k=m+1}^{M} \sum_{|y|\geq n^{1/2}\log n} Z_{k-1}^{x}[\1{X_k=y}]E_{k,y}\Big|\Big]\\
 &\leq C{n^{-d/2}} \sum_{k=m+1}^{M} \sum_{|y|\geq n^{1/2}\log n,|x|\leq Ln^{1/2}} p_k(x,y),
\end{align*}
where $L$ is chosen large enough that the support of $f$ is contained in $[-L,L]^d$. The expression above clearly decays super-polynomially. Similarly,
\begin{align*}
&\E\Big[\Big|n^{-d/2}\sum_{x\in\Z^d} f(x/\sqrt n)\sum_{k=m}^{M} \sum_{|y|>n^{1/2}\log n } \alpha_k(x,y)E_{k,y}\Big|\Big]\leq C{n^{1-d/2}}\max_{|x|\leq Ln^{1/2},|y|\geq n^{1/2}\log n, k\leq M}\mathbb{E}[\alpha_k(x,y)],
\end{align*}
which decays again super-polynomially due to Lemma~\ref{lem:compare}.

For the bulk contribution of the first term in \eqref{eq:aa}, applying Lemma~\ref{lem:bh} with {$I\coloneqq \llbracket m+1,M\rrbracket\times \llbracket - n^{1/2}\log n,n^{1/2}\log n\rrbracket^d$,}
$A_{k,y}\coloneqq \sum_{x\in\Z^d}f(x/\sqrt n)Z_{k-1}^{x}[\1{X_k=y}]$ 
gives
\begin{align*}
    &\E\Big[\Big|\sum_{x\in\Z^d} f(x/\sqrt n)\sum_{k=m+1}^{M} \sum_{|y|\leq n^{1/2}\log{n}} Z_{k-1}^{x}[\1{X_k=y}]E_{k,y}\Big|^p\Big]\\
    &\leq C\sum_{k=m+1}^M\sum_{|y|\leq  n^{1/2}\log n}\E\Big[\Big|\sum_{x\in\Z^d}f(x/\sqrt n)Z_{k-1}^{x}[\1{X_k=y}]\Big|^p\Big]\\
    &\leq C \|f\|_\infty \sum_{k=m+1}^M\sum_{|y|\leq  n^{1/2}\log n}\E\Big[\Big(\cev Z_{(0,k)}^{k,y}\Big)^p\Big].
\end{align*}
Hence, we see that for $p\in(1,\p\wedge 2)$, using $\mathbb{E}[|X|]\leq \mathbb{E}[|X|^p]^{1/p}$ and $\sup_{y\in\Z^d,k\in\N}\mathbb{E}[(\cev Z_{(0,k)}^{k,y})^p]<\infty$ by definition of $\p$, the left-hand side of \eqref{eq:aa} is bounded by $C n^{-\frac{d}{2}+\frac{d}{2p}}(M-m)^{1/p}(\log n)^{d/p}$. Since $M-m\leq n^{1-\delta}$ by assumption, the claim thus follows by choosing $p$ sufficiently close to $\p\wedge 2$. For \eqref{eq:bb}, we apply Lemma~\ref{lem:bh}  with 
\begin{align*}
	A_{k,y}^1&\coloneqq \sup_{|x|\leq Ln^{1/2}} (Z^x_{k-1})^{-1},\\
	A_{k,y}^2&\coloneqq \sum_{x\in\Z^d}f(x/\sqrt n)Z_{k-1}^{x}[\1{X_k=y}].
\end{align*}
We see that the left-hand side in \eqref{eq:bb} with the constraint $|y|\leq n^{1/2}\log n$ is bounded by
\begin{align*}
	C n^{-d/2} \E\Big[\sup_{|x|\leq Ln^{1/2},k\geq 0}(Z^x_k)^{-p'}\Big]^{1/p'}\Big(\sum_{k=m}^M\sum_{|y|\leq n^{1/2}\log n}\E\Big[\Big|\sum_{x\in\Z^d}f(x/\sqrt n)Z_{k-1}^{x}[\1{X_k=y}]\Big|^p\Big]\Big)^{1/p},
\end{align*}
where $\frac{1}{p}+\frac{1}{p'}=1$ with $p\in (1,\p\wedge 2)$. The second term is the same as before, and by \eqref{eq:union}, the first expectation is of order $O(n^{\varepsilon})$ for arbitrarily small $\varepsilon>0$.
\end{proof}

\section{Approximation of SHE}\label{sec:she}

{Recall that we have already proved that the contribution from small times to SHE is negligible, i.e., \eqref{eq:SHE_small}, in Lemma~\ref{lem: martingale estimate until Ln}. Thus,  it only remains to show that the contribution from large times is well-approximated by $M_n^{\delta}$. We assume \eqref{eq:assumption}, \eqref{eq:concprop}, and  $\delta\in (0,1/6)$ in this section.}
\begin{proof}[Proof of \eqref{eq:SHE_large}]
	We have
	\begin{equation}\label{eq:sdadsa}\begin{split}
	\S^\delta_n(f)-\M_n^{\delta}(f) &= 
			 n^{-d/2}\sum_{x\in\Z^d} f(x/\sqrt n)\sum_{k=n^{1-\delta}+1}^n \sum_{y\in\Z^d} p_k(x,y) (Z_{(0,k)}^{0,x;k,y}-\cev Z_{[k-k^{1/8},k)}^{k,y} )E_{k,y}.
\end{split}\end{equation}
The contribution from $|y|\geq n^{1/2}\log n$ can easily be shown to decay super-polynomially by a simple $L^1$-bound and \eqref{eq:well_known_quadr}, as in the proof of Lemma~\ref{lem: martingale estimate until Ln}. We now consider the remaining contribution in \eqref{eq:sdadsa}. Let $L$ be large enough that the support of $f$ is contained in $[-L,L]^d$. Let $p\in (1+d/2,\p\wedge 2)$. By using $(a+b)^p\leq 2 a^p + 2b^p$ for $a,b\geq 0$ and $p\in [1,2]$, we can compute
\begin{align}
	&\E\Big[\Big|\sum_{x\in\Z^d} f(x/\sqrt n)\sum_{k=n^{1-\delta}+1}^n \sum_{|y|\leq n^{1/2}\log n} p_k(x,y) (Z_{(0,k)}^{0,x;k,y}-\cev Z_{[k-k^{1/8},k)}^{k,y} )E_{k,y}\Big|^p\Big]\notag\\
	&\leq 2\E\Big[\Big|\sum_{k=n^{1-\delta}+1}^n \sum_{|y|\leq n^{1/2}\log n} \big(\cev Z_{[k-k^{1/8},k)}^{k,y}[f(X_0/\sqrt n)]-\cev Z_{(0,k)}^{k,y}[f(X_0/\sqrt n)]\big)E_{k,y}\Big|^p\Big]\notag\\
	&\qquad + 2\E\Big[\Big|\sum_{k=n^{1-\delta}+1}^n \sum_{|y|\leq n^{1/2}\log n} \Big( \sum_{x\in\Z^d} f(x/\sqrt n) p_k(x,y)\cev Z_{[k-k^{1/8},k)}^{k,y}-\cev Z_{[k-k^{1/8},k)}^{k,y}[f(X_0/\sqrt n)] \Big)E_{k,y}\Big|^p\Big].\label{eq:second_term}
\end{align}
Setting {$I\coloneqq \llbracket n^{1-\delta}+1,n\rrbracket\times\llbracket-n^{1/2}\log n,n^{1/2}\log n\rrbracket^d$ and}
\begin{align*}
    A_{k,y}\coloneqq (\cev Z_{[k-k^{1/8},k)}^{k,y}[f(X_0/\sqrt n)]-\cev Z_{(0,k)}^{k,y}[f(X_0/\sqrt n)]),
\end{align*}Lemma~\ref{lem:bh} shows that the first term is bounded by
\begin{align*}
    &C\sum_{k=n^{1-\delta}+1}^n \sum_{|y|\leq n^{1/2}\log n}\mathbb{E}[|\cev Z_{[k-k^{1/8},k)}^{k,y}[f(X_0/\sqrt n)]-\cev Z_{(0,k)}^{k,y}[f(X_0/\sqrt n)]|^p]\\
    &\leq C n^{\frac{2+d}2 }(\log n)^d n^{-(1-\delta)(\frac d2(p-1)+1)/8},
\end{align*}
where the last bound is due to Proposition~\ref{prop:diff}({i}).  For the second term, by \eqref{eq:ratio_srw}, we observe that, for any $|x|\leq Ln^{1/2}$, $k\in \Iintv{ n^{1-\delta},n}$, $|y|\leq n^{1/2}\log n$, and  $z\in\Z^d$ with $(0,x)\leftrightarrow(k,y)\leftrightarrow (k-k^{1/8},z)$,
\begin{align*}
	\Big|\frac{p_{k-k^{1/8}}(z-x)}{p_k(x,y)}-1\Big| 
    \leq {k^{-1/5}\leq n^{-1/6}}.
\end{align*}
Note that,
\begin{align*}
&\Big|\sum_{x\in\Z^d} f(x/\sqrt n) p_k(x,y)\cev Z_{[k-k^{1/8},k)}^{k,y}-\cev Z_{[k-k^{1/8},k)}^{k,y}[f(X_0/\sqrt n)]\Big|\\
&=\Big|\sum_{z\in \Z^d}\cev Z_{[k-k^{1/8},k)}^{k,y}[\1{X_{k-k^{1/8}}=z}]\Big(\sum_{x\in\Z^d}f(x/\sqrt n) \big(p_k(x,y)-p_{k-k^{1/8}}(x,z)\big)\Big)\Big|\\
&\leq Cn^{-{1/6}}\sum_{z\in \Z^d}\cev Z_{[k-k^{1/8},k)}^{k,y}[\1{X_{k-k^{1/8}}=z}]\Big(\sum_{x\in\Z^d}|f(x/\sqrt n)| p_k(x,y)\Big)\\
&\leq C\|f\|_\infty n^{-{1/6}} \cev Z_{[k-k^{1/8},k)}^{k,y}.
\end{align*} 
We use Lemma~\ref{lem:bh} with $A_{k,y}\coloneqq \sum_{x\in\Z^d} f(x/\sqrt n) p_k(x,y)\cev Z_{[k-k^{1/8},k)}^{k,y}-\cev Z_{[k-k^{1/8},k)}^{k,y}[f(X_0/\sqrt n)]$ to bound the {term in \eqref{eq:second_term}} by
\begin{align*}
    Cn^{-{p/6}}\sum_{k=n^{1-\delta}+1}^n \sum_{|y|\leq n^{1/2}\log n}\mathbb{E}[|\cev Z_{[k-k^{1/8},k)}^{k,y}|^p]\leq C n^{\frac{2+d}2-\frac p6}(\log n)^d.
\end{align*}
Now, to conclude, fix $p_0\in(1+ 2/d,\p\wedge 2)$ and set $\eps\coloneqq \max\{{p_0/6},(\frac d2(p_0-1)+1){/16}\}$. We have thus shown that for any $p\in(p_0,\p\wedge 2)$ there exists $C(p)$ such that,
\begin{align*}
    \E |\S^\delta_n(f)-\M_n^{\delta}(f) | \leq C(p) n^{-\frac{d}2+\frac{1+d/2}{p}-\frac{\eps}p} \leq  C(p) n^{-\frac{d}2+\frac{1+d/2}{p}-\frac{\eps}2}.
\end{align*}
The claim follows by choosing $p$ sufficiently close to $\p\wedge 2$ so that $-\frac{d}2+\frac{1+d/2}{p}\leq -\xi + \frac \eps 4$.
\end{proof}

\section{Approximation of KPZ}\label{sec:kpz}

We assume \eqref{eq:assumption} and \eqref{eq:concprop} throughout this section. Let $\s\coloneqq 1/(\p\wedge 2)$.
\subsection{Contribution from large times}
We show that the contribution from large times to KPZ is well-approximated by $M_n^\delta$. 
\begin{proof}[Proof of \eqref{eq:KPZ_large}]
Recall the definition of $\widehat K_n(x)$ in \eqref{eq:def_Khat}. We first observe that 
\begin{align*}
&\E\Big[\Big|\log\frac{Z_n^{x}}{Z_{n^{1-\delta}}^{x}}-(\widehat{K}_n(x)-\widehat{K}_{n^{1-\delta}}(x))\Big|\Big]\\
&=\E\Big[\Big|\sum_{k=n^{1-\delta}+1}^n \log \frac{Z_k^x}{Z_{k-1}^x}-(\widehat{K}_n(x)-\widehat{K}_{n^{1-\delta}}(x))\Big|\Big]\\
&=\E\Big[\Big|\sum_{k=n^{1-\delta}+1}^n {\Big(}\log\big(1+\sum_{y\in\Z^d} \alpha_k(x,y)E_{k,y}\big)-\sum_{y\in\Z^d}\alpha_k(x,y)  E_{k,y}{\Big)}\Big|\Big]\\
&\leq \sum_{k=n^{1-\delta}+1}^n\E\Big[\Big| \log\big(1+\sum_{y\in\Z^d} \alpha_k(x,y)E_{k,y}\big)-\sum_{y\in\Z^d}\alpha_k(x,y)E_{k,y}\Big|\Big]\\
&\leq C \sum_{k=n^{1-\delta}+1}^n\sum_{y\in\Z^d}\mathbb{E}\left[\alpha_k(x,y)^2\right] = C \sum_{k=n^{1-\delta}+1}^n\sum_{y\in\Z^d}\mathbb{E}\left[\alpha_k(y)^2\right],
\end{align*}
where we have used Proposition~\ref{prop:log} in the last line. In particular,
\begin{align*}
	\Big|\E\Big[\log \frac{Z_n}{Z_{n^{1-\delta}}}\Big]\Big|&=\Big|\E\Big[\log \frac{Z_n}{Z_{n^{1-\delta}}}-(\widehat{K}_n(0)-\widehat{K}_{n^{1-\delta}}(0))\Big]\Big|\\
							       &\leq\E\Big[\Big|\log \frac{Z_n}{Z_{n^{1-\delta}}}-(\widehat{K}_n(0)-\widehat{K}_{n^{1-\delta}}(0))\Big|\Big]\\
							       &\leq C \sum_{k=n^{1-\delta}+1}^n\sum_{y\in\Z^d}\mathbb{E}\left[\alpha_k(y)^2\right].
\end{align*}
{By \eqref{eq:partial_sum}, there exists $\delta_0\in (0,1/6)$ such that for any $\delta\in(0,\delta_0)$, the last line is smaller than $Cn^{-\xi-\eta}$ for some $\eta = \eta(\delta_0)>0$.} 

Next, let $p\in (1+d/2,\p\wedge 2)$. By \eqref{eq:burkholder1}, we estimate
	\begin{equation}\label{eq: bukholdser1}
	    \begin{split}
	         &\mathbb{E}[|(\widehat K_n(x)-\widehat K_{n^{1-\delta}}(x))-M_n^\delta(x)|]\\
	&=\E\Big[\Big|\sum_{k=n^{1-\delta}+1}^n\sum_{y\in\Z^d} \Big(\alpha_k(x,y)-p_k(x,y) \cev Z_{[k-k^{1/8},k)}^{k,y}\Big)E_{k,y}\Big|\Big]\\
 & \leq \E\Big[\sum_{k=n^{1-\delta}+1}^n\sum_{|y{-x}|>\sqrt{k}\log k} \Big|\alpha_k(x,y)-p_k(x,y) \cev Z_{[k-k^{1/8},k)}^{k,y}\Big|\Big]\\
 &\qquad + C(p)  \left(\sum_{k=n^{1-\delta}+1}^n\sum_{|y{-x}|\leq \sqrt{k}\log k} \E\Big[\Big|\alpha_k(x,y)-p_k(x,y) \cev Z_{[k-k^{1/8},k)}^{k,y}\Big|^{p} \Big]\right)^{1/p}.
 \end{split}
 \end{equation}
For $|y-x|>\sqrt{k}\log k$, by Lemma~\ref{lem:compare}, we have
\al{
\E\Big|\alpha_k(x,y)-p_k(x,y) \cev Z_{[k-k^{1/8},k)}^{k,y}\Big|\leq \E\Big[\alpha_k(x,y)+p_k(x,y) \cev Z_{[k-k^{1/8},k)}^{k,y}\Big] \leq C e^{-  (\log k)^{\gamma}/C} + {2 k p_k(x,y)},
}
which decays super-polynomially in $k.$ For $|y-x|\leq \sqrt{k}\log k$, by Lemma~\ref{lem:ratio},
 we have
 \al{
 \E\Big[\Big|\alpha_k(x,y)-p_k(x,y) \cev Z_{[k-k^{1/8},k)}^{k,y}\Big|^{p}\Big]= \E\Big[\Big|\alpha_k(y-x)-p_k(y-x) \cev Z_{[k-k^{1/8},k)}^{k,y-x}\Big|^{p}\Big]\leq C(p) k^{-\frac{p d}{2} - \eps},
 } 
where $\eps$ is independent of $p.$ Thus, if we take $\delta>0$ small enough and $p$ close enough to $\p\wedge 2$ depending on $\eps$, then we further {bound the last term in} \eqref{eq: bukholdser1} from above by
\al{
  \left(C \sum_{k=n^{1-\delta}+1}^n k^{-\frac{p d}{2} + \frac{d}{2}  - \eps} (\log k)^{d} \right)^{1/p} \leq C n^{(1-\delta) (-\frac{d}{2} + \frac{d+2}{2 p}  - \frac{\eps}{2 p})} \leq   C n^{-\xi -\frac{\eps}{8}}.
}
\end{proof}
\subsection{Contribution from small times}
{In this section, we prove that the contribution from small times to KPZ is negligible, i.e., \eqref{eq:KPZ_small}}. This approximation is rather technical, and we split the argument into several steps. First,  we introduce the martingale part and the previsible part of the increment $\log Z^x_k-\log Z^x_{k-1}$:
\begin{align}
	\Dprev_k(x)&\coloneqq \E\Big[\log{\frac{Z^x_k}{Z^x_{k-1}}}\Big|~\kF_{k-1}\Big]-\E\Big[\log \frac{Z_k}{Z_{k-1}}\Big],\label{eq:def_kpz_prev}\\
	\Dmg_k(x)&\coloneqq \log{\frac{Z^x_k}{Z^x_{k-1}}}-\mathbb{E}\left[\log{\frac{Z^x_k}{Z^x_{k-1}}}\Big|~\kF_{k-1}\right].\label{eq:def_mg_prev}
\end{align}
Observe that
\begin{align*}
	k_n^{\delta}(f)&=n^{-d/2}\sum_{x\in\Z^d}f(x/\sqrt n)\Big(\sum_{k=1}^{n^{1-\delta}}\log \frac{Z^x_k}{Z^x_{k-1}}-\E\Big[\log\frac{Z_k}{Z_{k-1}}\Big]\Big)\\
		      &=n^{-d/2}\sum_{x\in\Z^d}f(x/\sqrt n)\sum_{k=1}^{n^{1-\delta}} \Dmg_k(x)+n^{-d/2}\sum_{x\in\Z^d}f(x/\sqrt n)\sum_{k=1}^{n^{1-\delta}}\Dprev_k(x).
\end{align*}
Thus,  \eqref{eq:KPZ_small} follows from the following bounds:
\begin{prop}
{Let $\delta \in (0,1/6)$ be such that $\s < 1-\delta$}. There exist $\rho \in(0,\s\wedge(1-\delta - \s))$, $\eps>0$, {and $C>0$} such that, the following hold for  $n$ large enough:
\begin{align}
	\E\Big[\Big|n^{-d/2}\sum_{x\in\Z^d}f(x/\sqrt n)\sum_{k=1}^{n^{1-\delta}} \Dmg_k(x) \Big|\Big]\leq {C}n^{-\xi-\eps},\label{eq:KPZ_mg_small}\\
	\E\Big[\Big|n^{-d/2}\sum_{x\in\Z^d}f(x/\sqrt n)\sum_{k=1}^{n^{\s-\rho}} \Dprev_k(x) \Big|\Big]\leq {C}n^{-\xi-\eps},\label{eq:KPZ_prev_very_small}\\
	\E\Big[\Big|n^{-d/2}\sum_{x\in\Z^d}f(x/\sqrt n)\sum_{n^{\s-\rho}+1}^{n^{\s+\rho}} \Dprev_k(x) \Big|\Big]\leq {C}n^{-\xi-\eps},\label{eq:KPZ_prev_intermediate_small}\\
	\E\Big[\Big|{n^{-d/2}\sum_{x\in\Z^d}f(x/\sqrt n)}\sum_{n^{\s+\rho}+1}^{n^{1-\delta}} \Dprev_k(x) \Big|\Big]\leq {C} n^{-\xi-\eps}.\label{eq:KPZ_prev_not_too_small}
\end{align}
\end{prop}
The proof of \eqref{eq:KPZ_mg_small} is not too difficult, since the integrand is a martingale and there is good independence between the contributions from different starting positions. The approximation for the previsible part is split into three regimes, where the first case \eqref{eq:KPZ_prev_very_small} can still be treated using independence{, i.e., Proposition~\ref{prop: log correlation}(ii)}. The final case \eqref{eq:KPZ_prev_not_too_small} is also not too difficult by applying the control on the replica overlap from Proposition~\ref{prop:alpha2}. {The intermediate regime \eqref{eq:KPZ_prev_intermediate_small} is the most technical step, and we postpone the intuition to the beginning of Section~\ref{sec:intermediate}}. 
\subsection{Proof of \eqref{eq:KPZ_mg_small}, \eqref{eq:KPZ_prev_very_small} and \eqref{eq:KPZ_prev_not_too_small}}
The proof{s} of these claims {are} not that difficult.
\begin{proof}[Proof of \eqref{eq:KPZ_prev_not_too_small}]
We first note that there exists $C>0$ such that, for all $k\in\N$ and $x\in\Z^d$,
\begin{align}\label{Triangle bound}
\mathbb{E}[|\Dprev_k(x)|]\leq C\E\Big[\sum_{y\in \Z^d}\alpha_{k}(x,y)^2\Big].
\end{align}
Indeed, by the Jensen  inequality, 
\begin{align*}
	\mathbb{E}[|\Dprev_k(x)|]=\E\Big[\Big|\E\Big[\log\frac{Z^x_k}{Z^x_{k-1}}\Big|\F_{k-1}\Big]-\E\Big[\log\frac{Z_k}{Z_{k-1}}\Big]\Big|\Big]\leq 2\E\Big[\Big|\E\Big[\log\frac{Z^x_k}{Z^x_{k-1}}\Big|\F_{k-1}\Big]\Big|\Big],
\end{align*}
and by Proposition~\ref{prop:log}, there exists $C>0$ such that, 
\al{
	\Big|\E\Big[\log\frac{Z^x_k}{Z^x_{k-1}}\Big|\F_{k-1}\Big]\Big|=\Big|\E\Big[\log\Big(1+\sum_{y\in \Z^d}\alpha_k(x,y)E_{k,y}\Big)\Big|\F_{k-1}\Big]\Big|\leq C \sum_{y\in \Z^d}\alpha_{k}(x,y)^2.
}
{Now the claim follows from \eqref{eq:partial_sum}.}
\end{proof}
Next, we take care of the martingale part.
\begin{proof}[Proof of \eqref{eq:KPZ_mg_small}]
	We first take care of the contribution from times $k\leq n^{\s+\rho}$. By the Cauchy-Schwarz inequality, we have
\begin{align}
\mathbb{E}\Big[\Big|\sum_{x\in \Z^d}f(x/\sqrt{n})\sum_{k=1}^{n^{\s+\rho}}\Dmg_k(x)\Big|\Big]^2 &\leq  \mathbb{E}\Big[\Big(\sum_{x\in \Z^d}f(x/\sqrt{n})\sum_{k=1}^{n^{\s+\rho}}\Dmg_k(x)\Big)^2\Big]\notag\\
&=
\sum_{k=1}^{n^{\s+\rho}}\E\Big[\Big(\sum_{x\in\Z^d}  f\Big(\frac{x}{\sqrt{n}}\Big) \Dmg_k(x)\Big)^2\Big]\notag\\
 &=\sum_{k=1}^{n^{\s+\rho}} \sum_{x,x'\in\Z^d} f\Big(\frac{x}{\sqrt{n}}\Big)f\Big(\frac{x'}{\sqrt{n}}\Big)\mathbb{E}\left[\Dmg_k(x)\Dmg_k(x')\right],\label{eq:expression}
  \end{align}
  where in the second line we have used that the $\Dmg_k$ are orthogonal in $k$. We now apply Proposition~\ref{prop: log correlation}(ii) as follows: first, the contribution from terms with $k\leq n^\rho$ is bounded by $Cn^{\frac{d}{2}+\rho(d+1)}$. Next, for $k\geq n^{\rho}$ and $|x-x'|\geq k^{1/2}\log k$, Proposition~\ref{prop: log correlation}(ii) yields $|\mathbb{E}[\Dmg_k(x)\Dmg_k(x')]|\leq Ce^{-(\log k)^\gamma/C}\leq C e^{-\rho^\gamma ( \log n)^\gamma/C}$, so the contribution from such terms is negligible. To handle the remaining terms, by Proposition~\ref{prop: log correlation}(ii) again, we estimate 
  \begin{align*}
	&\Big|  \sum_{k=n^\rho+1}^{n^{\s+\rho}} \sum_{|x-x'|\leq k^{1/2}\log k} f(x/\sqrt n)f(x'/\sqrt n)\mathbb{E}[\Dmg_k(x)\Dmg_k(x')]\Big|\\
	&\leq C n^{\frac{d}{2}}\sum_{k=n^\rho+1}^{n^{\s+\rho}} k^{d -\frac{(\p\wedge 2)d}2+\eta}(\log k)^d\\
	&\leq C n^{\frac d2+(\s+\rho)( d+1-\frac{(\p\wedge 2)d}2+\eta)}(\log n)^d\\
	&=Cn^{d-2\xi-\s+\rho( d+1-\frac{(\p\wedge 2)d}2)+(\s+\rho)\eta} (\log n)^d.
  \end{align*}
Putting things together, if we choose $\rho$ and $\eta$ small enough that $\rho(d+1-\frac{(\p\wedge 2)d}2)+(\s+\rho)\eta<\frac{\s}{2}$ and $\rho(d+1)<1/2$, then for all $n$ large enough,
  \begin{align*}
	  \E\Big[\Big|n^{-d/2}\sum_{k=1}^{n^{\s+\rho}}f(x/\sqrt n)\Dmg_k(x)\Big|\Big]\leq C n^{-\frac{d-1}{4}} + Ce^{-( \log n)^\gamma /C}+ Cn^{-\xi-\frac{\s}{{4}}}(\log n)^{\frac{d}{2}}.
 \end{align*}
 Recalling that $\xi = \frac{d}{2} - \frac{d+2}{2(\p\wedge 2)}\leq \frac{d-2}4$, the above bound is indeed sufficient for \eqref{eq:KPZ_mg_small}. It remains to control the contribution from $k> n^{\s+\rho}$. To this end, with some constant $C>0$ by Proposition~\ref{prop:log},
\al{
{\sum_{k=n^{\s+\rho}+1}^n}\E\Big|\Dmg_k(x)- \sum_{y\in \Z^d}\alpha_{k}(x,y) E_{k,y}\Big|\leq C {\sum_{k=n^{\s+\rho}+1}^n}\E\sum_{y\in \Z^d}\alpha_{k}(x,y)^2,}
and the right-hand side {is bounded by $Cn^{-\xi-\eps}$ by \eqref{eq:partial_sum}, for some $\eps>0$ and $C>0$}. Finally, from Lemma~\ref{lem: martingale estimate until Ln}, we get 
\begin{align}\label{eq:eqeq}
	n^{-d/2}\E\Big[\Big|\sum_{x\in \Z^d}f(x/\sqrt n) \sum_{k=n^{\s+\rho}+1}^{n^{1-\delta}}\sum_{y\in \Z^d} \alpha_k(x,y) E_{k,y}\Big|\Big]\leq  n^{-\xi-\eps}.
\end{align}
Putting things together, the proof is completed.
\end{proof}
\begin{proof}[Proof of \eqref{eq:KPZ_prev_very_small}]
Using the Cauchy-Schwarz inequality, we have
\al{
\E\Big[\Big(\sum_{k=1}^{n^{\s-\rho}}\sum_{x\in\Z^d}  f\Big(\frac{x}{\sqrt{n}}\Big) \Dprev_k(x)\Big)^2\Big]
 &\leq { n^{\s-\rho}}\sum_{k=1}^{n^{\s-\rho}} \E\Big[\Big(\sum_{x\in\Z^d} f\Big(\frac{x}{\sqrt{n}}\Big) \Dprev_k(x)\Big)^2\Big]\\
 &={ n^{\s-\rho}}\sum_{k=1}^{n^{\s-\rho}} \sum_{x,x'\in\Z^d} f\Big(\frac{x}{\sqrt{n}}\Big)f\Big(\frac{x'}{\sqrt{n}}\Big)\mathbb{E}\left[\Dprev_k(x)\Dprev_k(x')\right].
}
This expression is {similar to the expression in} \eqref{eq:expression}, except that {the summation runs to $n^{\s-\rho}$ instead of $n^{\s+\rho}$,} ``$\Dmg$'' has been replaced by ``$\Dprev$'' {and the whole expression is multiplied by $n^{\s-\rho}$}. Since the bound from Proposition~\ref{prop: log correlation}(ii) applies to both quantities, the same argument as in the proof of \eqref{eq:KPZ_mg_small} {shows that}
\begin{align*}
&{n^{\s-\rho}\sum_{k=1}^{n^{\s-\rho}} \sum_{x,x'\in\Z^d} f\Big(\frac{x}{\sqrt{n}}\Big)f\Big(\frac{x'}{\sqrt{n}}\Big)\mathbb{E}\left[\Dprev_k(x)\Dprev_k(x')\right]}\\
&{\leq Cn^{\s-\rho} n^{d-2\xi-\s-\rho( d+1-\frac{(\p\wedge 2)d}2)+(\s-\rho)\eta} (\log n)^d,}
\end{align*}
{which is bounded by $n^{d-2\xi-\eps}$ for some $\eps>0$ if we choose $\eta$ small enough.}
\end{proof}

\subsection{Proof of \eqref{eq:KPZ_prev_intermediate_small}}\label{sec:intermediate}

{ This regime is the most difficult part of this section, as Propositions~\ref{prop:alpha2} and~\ref{prop: log correlation}(ii) do not yield the desired bound; indeed, both bounds become of order $n^{d-2\xi}$ as $\rho$ goes to $0$. Thus, we must understand correlations between log-partition functions with different starting times, while Proposition~\ref{prop: log correlation} only deals with correlations at the same starting time.}
\begin{lem}\label{lem: midterm}
There exist $\eta,C>0$ such that, the following holds. For any $n,m\in  \N$ with $m\geq n+2n^{1/8}$, we have for any $x,x'\in\Z^d$,
    \al{
    \mathbb{E}[|\Dprev_{n}(x)\Dprev_{m}(x')|]\leq C n^{\frac{d }{2}-\frac{(\p\wedge 2)d}{2}-\eta}.
    }
\end{lem}
We first complete the proof of \eqref{eq:KPZ_prev_intermediate_small}.
\begin{proof}[Proof of \eqref{eq:KPZ_prev_intermediate_small} assuming Lemma~\ref{lem: midterm}]  By expanding the square, we have
    \al{
&\E\Big[\Big|\sum_{x\in\Z^d} \sum_{k=n^{{\s-\rho}}}^{n^{\s+\rho}}f(x/\sqrt n) \Dprev_k(x)\Big|\Big]^2\\
 &\leq 2\sum_{x\in\Z^d}\sum_{x'\in\Z^d}   \sum_{k_1=n^{\s-\rho}}^{n^{\s+\rho}+1} \sum_{k_2=k_1}^{n^{\s+\rho}} \left|f\left(\frac{x}{\sqrt{n}}\right) f\left(\frac{x'}{\sqrt{n}}\right)\right|\left|\mathbb{E}\left[\Dprev_{k_1}(x)\Dprev_{k_2}(x')\right]\right|\\
 &= I_1+I_2+I_3,
}
where $I_3$ is defined by replacing the constraint ``$x'\in\Z^d$'' in the second sum by ``$|x-x'|> n^{{(\s+\rho)/2}}\log{n}$''. Moreover, from the remaining sum over $|x-x'|\leq n^{{(\s+\rho)/2}}\log{n}$, we define $I_1$ (resp. $I_2$) by taking the innermost sum over $k_2=k_1+2k_1^{{1/8}}+1,\dots,n^{\s+\rho}$ (resp. over  $k_2=k_1,\dots,k_1+2k_1^{{1/8}}+1$). Due to Proposition~\ref{prop: log correlation}(ii), we obtain that $I_3$ decays super-polynomially.
\smallskip By Lemma~\ref{lem: midterm}, each expectation in $I_1$ can be bounded by $k_1^{ \frac{d}{2}-(\p\wedge 2)\frac{d}{2}-\eta}$ with some $\eta>0$ independent of $\rho$. By considering the number of non-zero terms in each sum, we obtain
\begin{align*}
    I_1&\leq Cn^{\frac{d}{2}}\,( n^{\frac{(\s+\rho)d}{2}} (\log n)^d )\,n^{\s+\rho}\,n^{\s+\rho}\,n^{(\s-\rho)(\frac{d}{2}-\frac{(\p\wedge 2)d}{2}-\eta)}\\
    &= Cn^{\frac{d+2}{\p\wedge 2}+\rho(2+\frac{(\p\wedge 2)d}2)-\eta(\s-\rho)}\,(\log n)^d .
\end{align*}
Similarly, for $I_2$ we use Proposition~\ref{prop: log correlation}(ii) to bound the expectation as 
\begin{align*}
    |\mathbb{E}[\Dprev_{k_1}(x)\Dprev_{k_2}(x')]|\leq \mathbb{E}[\Dprev_{k_1}(x)^2]^{1/2}\mathbb{E}[\Dprev_{k_2}(x')^2]^{1/2}\leq n^{(\s-\rho)(\frac d2-\frac{(\p\wedge 2)d}2+\rho)},
\end{align*}
and thus 
\begin{align*}
     I_2&\leq Cn^{\frac{d}{2}}\,(n^{\frac{(\s+\rho)d}{2}}(\log n)^d )\,n^{\s+\rho}\,n^{\frac{\s+\rho}8}\,n^{(\s-\rho)(\frac{d}{2}-\frac{(\p\wedge 2)d}{2}+\rho)}\\
     &= Cn^{\frac{d+2}{\p\wedge 2}+\rho(\frac 98+\frac{(\p\wedge 2)d}2+\s-\rho)-\frac{7/8}{\p\wedge 2}}\, (\log n)^d.
\end{align*}
If we choose $\rho$ sufficiently small, we obtain $I_1+I_2\leq Cn^{2d-2\xi-\eps}$ for $\eps\coloneqq \min\{\frac{\eta(\s-\rho)}2,\frac{1}{2(\p\wedge 2)})$.
\end{proof}
\begin{proof}[Proof of Lemma~\ref{lem: midterm}]
We note that the  left-hand side  in the claim is bounded for fixed $n,m\in\N$, hence without loss of generality, we can suppose that $n$ is large enough, in particular, $n\geq 2^{{8/7}}$. 

Note that by by Proposition~\ref{prop:log}, there exists $C>0$ such that, for any $x,x',n$,
\aln{\label{Delta to alpha}
|\Dprev_{n}(x)\Dprev_{m}(x')|\leq C|\Dprev_{n}(x)| \Big(\sum_{y'\in \Z^d}\big(\alpha_{m}(x',y')^2+\mathbb{E}[\alpha_{m}(x',y')^2]\big)\Big).
}
Using again Proposition~\ref{prop:log} and Proposition~\ref{prop:alpha2}, we get
\begin{equation}\label{eq:thesame}
\begin{split}
    \mathbb{E}[|\Dprev_n(x)|]\sum_{y'}\E\Big[\alpha_m(x',y')^2\Big]&\leq C\Big(\sum_{y'}\E\Big[\alpha_n(x',y')^2\Big]\Big)\Big(\sum_{y'}\E\Big[\alpha_m(x',y')^2\Big]\Big)\\
    &\leq C n^{\frac d2-\frac{(\p\wedge 2)d}2+\eta}m^{\frac d2-\frac{(\p\wedge 2)d}2+\eta}\\
    &\leq Cn^{2(\frac d2-\frac{(\p\wedge 2)d}2+\eta)},
    \end{split}
    \end{equation}
where $\eta>0$ is arbitrary. The exponent in the last line is smaller than $\frac d2-\frac{(\p\wedge 2)d}2-\eta$ if $\eta$ is small enough. It remains to bound $\mathbb{E}[|\Dprev_n(x)|\sum_{y'}\alpha_m(x',y')^2]$, and for this purpose, we recall $\tilde{\alpha}$ from \eqref{eq:def_alphatilde}. Now, since both $\sum_{y'}\alpha_m(x',y')^2$ and $\widetilde\alpha_m(x',y')$ are bounded by $1$, we get, for any $\eta>0$,
\begin{equation}\label{eq:finalterm}
\begin{split}
&\Big|\sum_{y'\in\Z^d}\E\big[|\Dprev_n(x)|\alpha_m(x',y')^2]\Big|\\
    &\leq C\sum_{|x'-y'|\geq m^{1/2}\log m}\mathbb{E}[\alpha_m(x',y')^2]+\sum_{|x'-y'|\leq m^{1/2}\log m}\E\big[|\Dprev_n(x)|\widetilde \alpha_m(x',y')^2]\\
    &\quad+C\sum_{|x'-y'|\leq m^{1/2}\log m}  \mathbb{E}[|\alpha_{m}(x',y')^2-\tilde{\alpha}_{m}(x',y')^2|],
    \end{split}
    \end{equation}
where we have used that $\Dprev_m(x)$ is almost surely bounded due to Proposition~\ref{prop:log}. For the first term, we note that the sum is bounded by $\sum_{|x'-y'|\geq m^{1/2}\log m}\mathbb{E}[\alpha_m(x',y')]$, and by Lemma~\ref{lem:compare} this term decays super-polynomially in $m$. For the second term, we note that 
\begin{align*}
&\E\Big[\Big|\Dprev_n(x)\sum_{|x'-y'|\leq m^{1/2}\log m}\widetilde\alpha_m(x',y')^2\Big|\Big]\\
&=\mathbb{E}[|\Dprev_n(x)|]\Big(\sum_{|x'-y'|\leq m^{1/2}\log m}\mathbb{E}[\widetilde\alpha_m(x',y')^2]\Big)\\
&\leq C\Big(\sum_{y}\mathbb{E}[\alpha_n(x,y)^2]\Big)\Big(\sum_{y'} \mathbb{E}[\alpha_m(x',y')^2 ]\Big)+C\sum_{|x'-y'|\leq m^{1/2}\log m} \mathbb{E}[|\alpha_m(x',y')^2-\widetilde\alpha_m(x',y')^2|].
\end{align*}
In the equality, we have used that due to our assumption on $n$ and $m$, we had $n\leq m-m^{1/8}$. {To see that this indeed follows from our assumption $m\geq n + 2 n^{1/8}$, note that if $m\geq 2n$, then $n\leq m- (m/2)\leq m-m^{1/8}$. Otherwise, if $m\leq 2n$, then $n \leq m - 2n^{1/8} \leq m- 2(m/2)^{1/8}\leq m-m^{1/8}$. Therefore, $n\leq m-m^{1/8}$ holds in both cases and we conclude that $\Dprev_n(x)$ and $\widetilde \alpha_m(x',y')$ are independent. Now,} the first term {in the previous display} is bounded by $Cn^{d-(\p\wedge 2)d+\eta}$ due to Proposition~\ref{prop:alpha2}, which is the same bound as in \eqref{eq:thesame}. For the second term, which is the same as the final term in \eqref{eq:finalterm}, we apply Propositions~\ref{lem:ratio} and~\ref{prop:alpha2} to get, for fixed $\eps>0$ small enough,
\begin{align*}
\mathbb{E}[|\alpha_{m}(x',y')^2-\tilde{\alpha}_{m}(x',y')^2|]&=\mathbb{E}[(\alpha_{m}(x',y')+\tilde{\alpha}_{m}(x',y'))((\alpha_{m}(x',y')-\tilde{\alpha}_{m}(x',y'))]\\
    &\leq 2\mathbb{E}[\alpha_{m}(x',y')^2+\tilde{\alpha}_{m}(x',y')^2]^{1/2}\mathbb{E}[(\alpha_{m}(x',y')-\tilde{\alpha}_{m}(x',y'))^2]^{1/2}\\
    &= 2\mathbb{E}[\alpha_{m}(x',y')^2+\tilde{\alpha}_{m}(x',y')^2]^{1/2}\mathbb{E}[(\alpha_{m}(y'-x')-\tilde{\alpha}_{m}(y'-x'))^2]^{1/2}\\
        &\leq Cm^{-\frac{(\p\wedge 2)d}{4}+\eps}m^{-\frac{(\p\wedge 2)d}{4}-2\eps} = Cm^{-\frac{(\p\wedge 2)d}{2}-\eps}.
\end{align*}
By summing this over $|y'-x'|\leq m^{1/2}\log m$, we obtain the desired bound. 
\end{proof}
\appendix
\section{Doob's decomposition of $\log Z_n$}

The following result is essentially the same as \cite[Lemma 2.1]{CSY03}. Since we slightly change the presentation, we {repeat the proof here for completeness}.

\begin{prop}\label{prop:log}
Let $A$ be a countable set and $(\eta_i)_{i\in A}$ i.i.d.\ positive random variables with $\E \eta_i =1$. We assume that 
$$\mathbb{E}\left[\eta_i^2 + (\log{\eta_i})^2\right]<\infty.$$
We consider a probability measure $(a_i)_{i\in A}$ on A, i.e., $a_i \geq 0$ and  $\sum_{i \in A} a_i = 1.$ We define
$$U \coloneqq  \sum_{i\in A} a_i (\eta_i - 1),\quad \gamma \coloneqq  -\mathbb{E}[\log(\eta_i)].$$
Define $\phi(x)\coloneqq x-\log{(1+x)}$. There exists $C>0$ independent of $(a_i)$ such that,
\aln{\label{upper ineq for U}
\begin{split}
    &\mathbb{E}[|\phi(U)|]\leq C \sum_{i\in A}a_i^2,\\
&\mathbb{E}[(\log{(1+U)})^2]\leq C \sum_{i\in A}a_i^2.
\end{split}
}
\end{prop}
\begin{proof}
   Let    $V \coloneqq  \sum_{i\in A} a_i (\log{\eta_i} + \gamma).$ Let $\varepsilon > 0$ be such that, 
$$-\log\varepsilon - \gamma \geq 1.$$
\text{If } $U \leq \varepsilon-1,$ \text{ then}
    $$1+\gamma\leq -\log{\eps}\leq -\log(1+U) = -\log\left(\sum_{i\in A} a_i \eta_i\right).$$
    By the Jensen  inequality, this is further bounded from above by 
    $$-\sum_{i\in A} a_i \log \eta_i = -V + \gamma.$$
    Hence, we have 
    $$\{U \leq \varepsilon-1\} \subseteq \{-\log(1+U) \leq -V+\gamma\}
\cap 
\{1 \leq -V\}.
$$
Therefore, since $ (-V+\gamma)\leq (1+|\gamma|) V^2$ on the event $\{1 \leq -V\}$, we have
\al{
-\mathbb{E}[\log(1+U)\mathbf{1}\{U \leq \eps-1\}]& \leq \mathbb{E}[(-V+\gamma)\mathbf{1}\{1 \leq -V\}]\\
&\leq \mathbb{E}[(1+|\gamma|)  V^2]\\
&=(1+|\gamma|)\mathbb{E}[(\log{\eta_i}+\gamma)^2] \sum_{i\in A}a_i^2.
}
Similarly, { since $-\log(1+U)\geq -\log\varepsilon\geq 0$ on the event $\{U \leq \eps-1\}$}, we have
\al{
\mathbb{E}[(\log(1+U))^2\mathbf{1}\{U \leq \eps-1\}]& \leq \mathbb{E}[(-V+\gamma)^2 \{1 \leq -V\}]\\
&\leq (1+|\gamma|)^2 \mathbb{E}[V^2]=(1+|\gamma|)^2\mathbb{E}[(\log{\eta_i}+\gamma)^2] \sum_{i\in A}a_i^2.
}
Note that for $x>-1$, $\phi(x)\geq 0$. 
Note that there exists $C=C(\eps)>0$ such that, for $x\geq \eps-1$, $\log{(1+x)}\geq x- C x^2.$ Since $\mathbb{E}[U]=0$, $\mathbb{E}[U \mathbf{1}_{U\leq \eps-1}]\leq 0$. Thus, we have
\al{
\mathbb{E}[|\phi(U)|]=\E [\phi(U)]&= \mathbb{E}[(U-\log{(1+U)})\mathbf{1}_{U\geq \eps-1}]+\mathbb{E}[(U-\log{(1+U)})\mathbf{1}_{U\leq \eps-1}]\\
&\leq  C \mathbb{E}[U^2 \mathbf{1}_{U\geq \eps-1}]-\mathbb{E}[(\log{(1+U)})\mathbf{1}_{U\leq \eps-1}]\\
&\leq (C\mathbb{E}[(\eta_i-1)^2]+(1+|\gamma|)\mathbb{E}[(\log{\eta_i}+\gamma)^2]) \sum_{i\in A}a_i^2.
}
Similarly, since $\log{(1+x)}\leq C|x|$ for any $x\geq \e-1$ with some $C=C(\e)>0$, we have
\begin{align*}
\E [(\log{(1+U)})^2]&= \mathbb{E}[ (\log{(1+U)})^2 \mathbf{1}_{U\geq \eps-1}]+\mathbb{E}[(\log{(1+U)})^2 \mathbf{1}_{U\leq \eps-1}]\\
&\leq  C \mathbb{E}[U^2 \mathbf{1}_{U\geq \eps-1}]+\mathbb{E}[(\log{(1+U)})^2 \mathbf{1}_{U\leq \eps-1}]\\
&\leq (C\mathbb{E}[(\eta_i-1)^2]+(1+|\gamma|)^2\mathbb{E}[(\log{\eta_i}+\gamma)^2]) \sum_{i\in A}a_i^2.\qedhere
\end{align*}
\end{proof}

{
\section*{Acknowledgments}
This work was supported by JSPS Grant-in-Aid for Scientific Research 22K20344 and 23K12984. We are grateful to anonymous referees for their careful reading of our manuscript and for the constructive comments and suggestions. 
}

\bibliographystyle{plain}
\bibliography{ref}

\begin{thebibliography}{10}

\bibitem{AKQ14}
Tom Alberts, Konstantin Khanin, and Jeremy Quastel.
\newblock The continuum directed random polymer.
\newblock {\em J. Stat. Phys.}, 154(1):305--326, 2014.

\bibitem{amir2011probability}
Gideon Amir, Ivan Corwin, and Jeremy Quastel.
\newblock Probability distribution of the free energy of the continuum directed
  random polymer in 1+ 1 dimensions.
\newblock {\em Comm. Pure Appl. Math.}, 64(4):466--537, 2011.

\bibitem{ben2009large}
Iddo Ben-Ari.
\newblock Large deviations for partition functions of directed polymers in an
  {IID} field.
\newblock {\em Ann. Inst. Henri Poincar\'{e} Probab. Stat.}, 45(3):770--792,
  2009.

\bibitem{bertini1998two}
Lorenzo Bertini and Nicoletta Cancrini.
\newblock The two-dimensional stochastic heat equation: renormalizing a
  multiplicative noise.
\newblock {\em J. Phys. A: Math. Gen.}, 31(2):615, 1998.

\bibitem{bertini1997stochastic}
Lorenzo Bertini and Giambattista Giacomin.
\newblock Stochastic burgers and kpz equations from particle systems.
\newblock {\em Comm. Math. Phys.}, 183:571--607, 1997.

\bibitem{B04}
Matthias Birkner.
\newblock A condition for weak disorder for directed polymers in random
  environment.
\newblock {\em Electron. Comm. Probab.}, 9:22--25, 2004.

\bibitem{bolthausen1989note}
Erwin Bolthausen.
\newblock A note on the diffusion of directed polymers in a random environment.
\newblock {\em Comm. Math. Phys.}, 123:529--534, 1989.

\bibitem{B66}
D.~L. Burkholder.
\newblock Martingale transforms.
\newblock {\em Ann. Math. Statist.}, 37:1494--1504, 1966.

\bibitem{caravenna2019moments}
Francesco Caravenna, Rongfeng Sun, and Nikos Zygouras.
\newblock On the moments of the (2+ 1)-dimensional directed polymer and
  stochastic heat equation in the critical window.
\newblock {\em Comm. Math. Phys.}, 372(2):385--440, 2019.

\bibitem{CSZ20}
Francesco Caravenna, Rongfeng Sun, and Nikos Zygouras.
\newblock The two-dimensional {KPZ} equation in the entire subcritical regime.
\newblock {\em Ann. Probab.}, 48(3):1086--1127, 2020.

\bibitem{caravenna2023critical}
Francesco Caravenna, Rongfeng Sun, and Nikos Zygouras.
\newblock The critical 2{D} stochastic heat flow.
\newblock {\em Invent. Math.}, 233(1):325--460, 2023.

\bibitem{CTT17}
Francesco Caravenna, Fabio~Lucio Toninelli, and Niccol\`o Torri.
\newblock Universality for the pinning model in the weak coupling regime.
\newblock {\em Ann. Probab.}, 45(4):2154--2209, 2017.

\bibitem{chatterjee2020constructing}
Sourav Chatterjee and Alexander Dunlap.
\newblock Constructing a solution of the {$(2+1)$}-dimensional {KPZ} equation.
\newblock {\em Ann. Probab.}, 48(2):1014--1055, 2020.

\bibitem{C17}
Francis Comets.
\newblock {\em Directed polymers in random environments}, volume 2175 of {\em
  Lecture Notes in Mathematics}.
\newblock Springer, Cham, 2017.
\newblock Lecture notes from the 46th Probability Summer School held in
  Saint-Flour, 2016.

\bibitem{comets2017rate}
Francis Comets and Quansheng Liu.
\newblock Rate of convergence for polymers in a weak disorder.
\newblock {\em J. Math. Anal. Appl.}, 455(1):312--335, 2017.

\bibitem{CSY03}
Francis Comets, Tokuzo Shiga, and Nobuo Yoshida.
\newblock Directed polymers in a random environment: path localization and
  strong disorder.
\newblock {\em Bernoulli}, 9(4):705--723, 2003.

\bibitem{CY06}
Francis Comets and Nobuo Yoshida.
\newblock Directed polymers in random environment are diffusive at weak
  disorder.
\newblock {\em Ann. Probab.}, 34(5):1746--1770, 2006.

\bibitem{CN21}
Cl\'{e}ment Cosco and Shuta Nakajima.
\newblock Gaussian fluctuations for the directed polymer partition function in
  dimension {$d \geq 3$} and in the whole {$L^2$}-region.
\newblock {\em Ann. Inst. Henri Poincar\'{e} Probab. Stat.}, 57(2):872--889,
  2021.

\bibitem{cosco2022law}
Cl{\'e}ment Cosco, Shuta Nakajima, and Makoto Nakashima.
\newblock Law of large numbers and fluctuations in the sub-critical and {$L^2$}
  regions for {SHE} and {KPZ} equation in dimension $d\geq 3$.
\newblock {\em Stochastic Process. Appl.}, 151:127--173, 2022.

\bibitem{dunlap20232d}
Alexander Dunlap and Cole Graham.
\newblock The 2{D} nonlinear stochastic heat equation: pointwise statistics and
  the decoupling function.
\newblock {\em arXiv preprint arXiv:2308.11850}, 2023.

\bibitem{dunlap2022forward}
Alexander Dunlap and Yu~Gu.
\newblock A forward-backward {SDE} from the 2{D} nonlinear stochastic heat
  equation.
\newblock {\em Ann. Probab.}, 50(3):1204--1253, 2022.

\bibitem{dunlap2020fluctuations}
Alexander Dunlap, Yu~Gu, Lenya Ryzhik, and Ofer Zeitouni.
\newblock Fluctuations of the solutions to the {KPZ} equation in dimensions
  three and higher.
\newblock {\em Probab. Theory Relat. Fields}, 176(3):1217--1258, 2020.

\bibitem{gu2020gaussian}
Yu~Gu.
\newblock Gaussian fluctuations from the 2{D} {KPZ} equation.
\newblock {\em Stochastics Partial Differ. Equations-Anal. Comput.},
  8:150--185, 2020.

\bibitem{gu2021moments}
Yu~Gu, Jeremy Quastel, and Li-Cheng Tsai.
\newblock Moments of the 2{D} {SHE} at criticality.
\newblock {\em Probab. Math. Phys.}, 2(1):179--219, 2021.

\bibitem{gu2018edwards}
Yu~Gu, Lenya Ryzhik, and Ofer Zeitouni.
\newblock The edwards--wilkinson limit of the random heat equation in
  dimensions three and higher.
\newblock {\em Comm. Math. Phys.}, 363:351--388, 2018.

\bibitem{gubinelli2015paracontrolled}
Massimiliano Gubinelli, Peter Imkeller, and Nicolas Perkowski.
\newblock Paracontrolled distributions and singular pdes.
\newblock {\em Forum Math. Pi}, 3, August 2015.

\bibitem{hairer2014theory}
Martin Hairer.
\newblock A theory of regularity structures.
\newblock {\em Invent. Math.}, 198(2):269--504, 2014.

\bibitem{HT15}
Timothy Halpin-Healy and Kazumasa~A Takeuchi.
\newblock A {KPZ} cocktail-shaken, not stirred... toasting 30 years of
  kinetically roughened surfaces.
\newblock {\em J. Stat. Phys.}, 160:794--814, 2015.

\bibitem{huse1985pinning}
David~A Huse and Christopher~L Henley.
\newblock Pinning and roughening of domain walls in ising systems due to random
  impurities.
\newblock {\em Phys. Rev. Lett.}, 54(25):2708, 1985.

\bibitem{imbrie1988diffusion}
John~Z Imbrie and Thomas Spencer.
\newblock Diffusion of directed polymers in a random environment.
\newblock {\em J. Stat. Phys.}, 52:609--626, 1988.

\bibitem{junk2023local}
Stefan Junk.
\newblock Local limit theorem for directed polymers beyond the {$L^2$}-phase,
  2023.
\newblock arXiv:2307.05097.

\bibitem{J22}
Stefan Junk.
\newblock Fluctuations of partition functions of directed polymers in weak
  disorder beyond the {$L^2$}-phase.
\newblock {\em Ann. Probab.}, 53(2):557--596, March 2025.

\bibitem{junk2024strongdisorderstrongdisorder}
Stefan Junk and Hubert Lacoin.
\newblock Strong disorder and very strong disorder are equivalent for directed
  polymers, 2024.
\newblock arXiv:2402.02562.

\bibitem{JL25_1}
Stefan Junk and Hubert Lacoin.
\newblock Coincidence of critical points for directed polymers for general
  environments and random walks, February 2025.
\newblock arXiv:2502.04113.

\bibitem{junk2024tail}
Stefan Junk and Hubert Lacoin.
\newblock The tail distribution of the partition function for directed polymers
  in the weak disorder phase.
\newblock {\em Comm. Math. Phys.}, 406(3), February 2025.

\bibitem{PhysRevLett.56.889}
Mehran Kardar, Giorgio Parisi, and Yi-Cheng Zhang.
\newblock Dynamic scaling of growing interfaces.
\newblock {\em Phys. Rev. Lett.}, 56:889--892, Mar 1986.

\bibitem{LL10}
Gregory~F. Lawler and Vlada Limic.
\newblock {\em Random walk: a modern introduction}, volume 123 of {\em
  Cambridge Studies in Advanced Mathematics}.
\newblock Cambridge University Press, Cambridge, 2010.

\bibitem{L01}
Michel Ledoux.
\newblock {\em The concentration of measure phenomenon}, volume~89 of {\em
  Mathematical Surveys and Monographs}.
\newblock American Mathematical Society, Providence, RI, 2001.

\bibitem{lygkonis2022edwards}
Dimitris Lygkonis and Nikos Zygouras.
\newblock Edwards–wilkinson fluctuations for the directed polymer in the full
  {L}$^2$-regime for dimensions $d\geq 3$.
\newblock {\em Ann. Inst. Henri Poincar\'{e}, Probab. Stat.}, 58(1), February
  2022.

\bibitem{magnen2018scaling}
Jacques Magnen and J{\'e}r{\'e}mie Unterberger.
\newblock The scaling limit of the {KPZ} equation in space dimension 3 and
  higher.
\newblock {\em J. Stat. Phys.}, 171:543--598, 2018.

\bibitem{mukherjee2016weak}
Chiranjib Mukherjee, Alexander Shamov, and Ofer Zeitouni.
\newblock Weak and strong disorder for the stochastic heat equation and
  continuous directed polymers in $d\geq 3$.
\newblock {\em Electron. Commun. Probab.}, 21(none), January 2016.

\bibitem{nakajima2023fluctuations}
Shuta Nakajima and Makoto Nakashima.
\newblock Fluctuations of two-dimensional stochastic heat equation and kpz
  equation in subcritical regime for general initial conditions.
\newblock {\em Electron. J. Probab.}, 28:1--38, 2023.

\bibitem{sasamoto2010one}
Tomohiro Sasamoto and Herbert Spohn.
\newblock One-dimensional kardar-parisi-zhang equation: an exact solution and
  its universality.
\newblock {\em Phys. Rev. Lett.}, 104(23):230602, 2010.

\bibitem{zygouras2024directedpolymersrandomenvironment}
Nikos Zygouras.
\newblock Directed polymers in a random environment: A review of the phase
  transitions.
\newblock {\em Stochastic Process. Appl.}, 177:104431, November 2024.

\end{thebibliography}

\end{document}